\documentclass[11pt,a4paper, reqno]{amsart}

\usepackage{a4wide}

\usepackage{amsfonts,amsthm,amssymb,amsmath,amscd,amsopn,mathabx,mathrsfs}
\usepackage[pdftex]{graphicx}
\usepackage{graphics}
\usepackage[matrix,arrow]{xy}
\usepackage[active]{srcltx}
\usepackage[draft]{todonotes}
\usepackage{tensor}

\usepackage[colorlinks=true]{hyperref}
\hypersetup{
colorlinks   = true,
citecolor    = blue
}

\usepackage{enumitem, hyperref}\hypersetup{colorlinks}

\makeatletter
\def\reflb#1#2{\begingroup
    #2%
    \def\@currentlabel{#2}%
    \phantomsection\label{#1}\endgroup
}
\makeatother

\usepackage{color}

\definecolor{darkred}{rgb}{1,0,0} 
\definecolor{darkgreen}{rgb}{0,0.8,0}
\definecolor{darkblue}{rgb}{0,0,1}

\hypersetup{colorlinks,
linkcolor=darkblue,
filecolor=darkgreen,
urlcolor=darkred,
citecolor=darkblue}

\newtheorem{thm}{Theorem}
\numberwithin{thm}{section}
\numberwithin{equation}{section}
\newtheorem{theorem}[thm]{Theorem}
\newtheorem*{theorem*}{Theorem}

\newtheorem{corollary}[thm]{Corollary}
\newtheorem*{corollary*}{Corollary}
\newtheorem{lemma}[thm]{Lemma}

\newtheorem{proposition}[thm]{Proposition}

\newtheorem*{conjecture*}{Conjecture}

\newtheorem*{question*}{Question}

\newtheorem*{definition*}{Definition}

\newtheorem*{definitions*}{Definitions}

\newtheorem*{rem*}{Remark}
\theoremstyle{remark}
\newtheorem{remark}[thm]{Remark}

\newtheorem*{remark*}{Remark}

\newtheorem*{remarks*}{Remarks}
\newtheorem*{example*}{Example}

\newtheorem*{examples*}{Examples}

\newcommand{\R}{\mathbb{R}}
\newcommand{\Z}{\mathbb{Z}}
\newcommand{\Q}{\mathbb{Q}}
\newcommand{\C}{\mathbb{C}}
\newcommand{\N}{\mathbb{N}}

\def\CP{{\mathbb C}P}
\def\RP{{\mathbb R}P}
\def\HP{{\mathbb H}P}
\def\CaP{{\text Ca}P}
\def\PP{{\mathbb P}}

\newcommand{\ep}{\epsilon}
\newcommand{\ga}{\gamma}
\newcommand{\Ga}{\Gamma}

\newcommand{\vr}{\varphi}
\newcommand{\om}{\omega}
\newcommand{\Om}{\Omega}

\newcommand    \TSp     {\widetilde{\mathrm{Sp}}}

\newcommand{\SO}{\mathrm{SO}}

\newcommand\numberthis{\addtocounter{equation}{1}\tag{\theequation}}

\def\lg{\langle}
\def\rg{\rangle}

\def\cz{{\mu}}

\def\PP{{\mathscr P}}
\def\HC{{\mathrm{HC}}}
\def\SH{{\mathrm{SH}}}
\def\CC{{\mathrm{CC}}}
\def\HF{{\mathrm{HF}}}
\def\H{{\mathrm{H}}}

\def\L{L^{2n+1}_p(\ell_0,\dots,\ell_n)}
\def\Lm{L^{m}_p(\ell_0,\dots,\ell_{d-1})}

\def\halpha{{\hat\alpha}}

\def\Im{{\text{Im}}}

\def\hga{{\hat\gamma}}
\def\mi{{\hat\mu}}
\def\M{{\widebar M}}

\newcommand{\p}{\partial}
\def\vk{\vec{k}}
\def\Nov{{\Lambda}}
\def\hH{{\hat H}}
\def\kmin{{k_{\text min}}}
\def\sec{{\mathfrak s}}
\def\bxi{{\widebar\xi}}

\begin{document}

\title[Periodic orbits of non-degenerate lacunary contact forms]{Periodic orbits of non-degenerate lacunary\\ contact forms on prequantization bundles}

\author[Miguel Abreu]{Miguel Abreu}
\author[Leonardo Macarini]{Leonardo Macarini}

\address{Center for Mathematical Analysis, Geometry and Dynamical Systems,
Instituto Superior T\'ecnico, Universidade de Lisboa, 
Av. Rovisco Pais, 1049-001 Lisboa, Portugal}
\email{mabreu@math.tecnico.ulisboa.pt}

\address{IMPA,
Estrada Dona Castorina, 110, Rio de Janeiro, 22460-320, Brazil}
\email{leonardo@impa.br}


\thanks{MA was partially funded by FCT/Portugal through UID/MAT/04459/2020. LM was partially supported by FAPERJ and CNPq/Brazil.}

\begin{abstract}
A non-degenerate contact form is lacunary if the indexes of every contractible periodic Reeb orbit have the same parity. To the best of our knowledge, every contact form with finitely many periodic orbits known so far is non-degenerate and lacunary. We show that every non-degenerate lacunary contact form on a suitable prequantization of a closed symplectic manifold $B$ has precisely $r_B$ contractible closed orbits, where $r_B=\dim \H_*(B;{\mathbb Q})$. Examples of such prequantizations include the standard contact sphere, the unit cosphere bundle of a compact rank one symmetric space (CROSS) and many others. We also consider some prequantizations of orbifolds, like lens spaces and the unit cosphere bundle of lens spaces, and obtain multiplicity results for these prequantizations.
\end{abstract}

\maketitle

\tableofcontents

\section{Introduction}

\subsection{Introduction and Main Results}

Let $(M^{2n+1},\xi)$ be a closed co-oriented contact manifold. Let $\alpha$ be a contact form supporting $\xi$ (i.e. such that $\ker \alpha=\xi$) and denote by $R_\alpha$ the corresponding Reeb vector field uniquely characterized by the equations $\iota_{R_\alpha} d\alpha = 0$ and $\alpha (R_\alpha) = 1$. Reeb flows form a prominent class of Hamiltonian systems on regular energy levels. Indeed, if $M$ is a contact type hypersurface in a symplectic manifold $W$ and $H: W \to \R$ is a Hamiltonian such that $M$ is a regular energy level of $H$, then the Hamiltonian flow of $H$ on $M$ is a reparametrization of the Reeb flow. There are several important examples, including proper homogeneous Hamiltonians $H: \R^{2n} \to \R$ and geodesic flows.

In this work, we will address the problem of the multiplicity of periodic orbits of Reeb flows, that is, the number of simple (i.e. non-iterated) periodic Reeb orbits. Two key tools to attack this problem are \emph{(positive) equivariant symplectic homology} and \emph{linearized contact homology}. The latter was introduced by Eliashberg, Givental and Hofer in their seminal paper \cite{EGH}. The first was introduced by Viterbo \cite{Vit} and developed by Bourgeois and Oancea \cite{BO09,BO10, BO13a, BO13b, BO17}. As proved in \cite{BO17}, these homologies are isomorphic whenever linearized contact homology is well defined.

When the linearized contact homology (with rational coefficients) is unbounded, that is, the dimension of the corresponding vector spaces goes to infinity for some sequence of degrees $k_i \to \infty$ we have that every contact form on $M$ has infinitely many simple closed orbits \cite{HM}; cf. \cite{McL1}. We say that this is the \emph{homologically unbounded} case; see the survey \cite{Mac}.

Thus, in the multiplicity problem the interesting case is the \emph{homologically bounded} one. This case is much more involved since, in this context, we do have examples of Reeb flows with finitely many simple periodic orbits. Let us consider a prequantization $S^1$-bundle $(M^{2n+1},\xi)$ of a closed integral symplectic manifold $(B,\om)$. This can be considered as a prototypical example of a homologically bounded contact manifold. As a matter of fact, under suitable hypotheses on $B$, its positive equivariant symplectic homology is given by a sum of copies of the singular homology of the basis with a shift in the degree; see, for instance, \cite{AM1,AM2,GGM1,GGM2} and Section \ref{sec:esh_preq}.

When $\om$ is aspherical, that is, $\om|_{\pi_2(B)}=0$, it was proved in \cite{GGM1,GGM2}, under minor extra assumptions on $M$ (probably just technical), that every contact form on $M$ has infinitely simple closed orbits; see also \cite{GS}. If $\om$ is not aspherical the problem is more delicate since there are examples of contact forms with finitely many closed orbits. Indeed, the standard contact sphere $S^{2n+1}$ is a prequantization of $\CP^n$ and it has contact forms with precisely $n+1$ simple closed orbits given by the irrational ellipsoids. More generally, every prequantization of a closed symplectic manifold admitting a Hamiltonian circle action with isolated fixed points has a contact form with finitely many simple closed orbits, and all these symplectic manifolds are necessarily not aspherical. Thus, if $\om$ is not aspherical we have, in general, to obtain a, ideally sharp, lower bound for the number of simple closed orbits.

A result in this direction is the following theorem proved by Ginzburg, G\"urel and Macarini \cite{GGM2}. In order to state it, we need some preliminary definitions. A closed Reeb orbit $\ga$ is \emph{non-degenerate} if its linearized Poincar\'e map does not have eigenvalue one. A contact form $\alpha$ is non-degenerate if every closed orbit $\ga$ of $\alpha$ is non-degenerate. Given a closed Reeb orbit $\ga$, let $\cz(\ga)$ and $\mi(\ga)$ be its Conley-Zehnder index and mean index respectively (see, for instance, \cite{SZ}). A contact form $\alpha$ supporting $\xi$ is \emph{index-positive} (resp.\ \emph{index-negative}) if the mean index $\mi(\ga)$ is positive (resp.\ negative) for every contractible periodic orbit $\ga$ of $\alpha$. We say that $\alpha$ is \emph{index-admissible} if every contractible closed orbit $\ga$ of $\alpha$ satisfies $\cz(\ga)>3-n$. Given a symplectic manifold $B$, let
\[ 
c_B := \inf \{k \in \N \mid \exists S \in \pi_2(B) \text{ with } \langle c_1(TB),S \rangle = k\} 
\]
be its minimal Chern number.

\begin{theorem}[\cite{GGM2}]
\label{thm:GGM}
Let $(M^{2n+1},\xi)$ be a prequantization $S^1$-bundle of a closed symplectic manifold $(B,\om)$ such that $\om|_{\pi_2(B)}\neq 0$, $c_B>n/2$ and $\H_{k}(B;\Q)=0$ for every odd $k$. Let $\alpha$ be a non-degenerate contact form supporting $\xi$ which is index-positive and has no contractible periodic orbits $\ga$ such that $\cz(\ga)=0$ if $n$ is odd or $\cz(\ga) \in \{0,\pm 1\}$ if $n$ is even. Assume that one of the following two conditions holds:
\begin{itemize}
\item[\reflb{cond:F1}{(F)}]  $M$ admits a strong symplectic filling $(W,\Om)$ such that $\Om|_{\pi_2(W)}=0$ and $c_1(TW)|_{\pi_2(W)}=0$, and the map $\pi_1(M) \to \pi_1(W)$ induced by the inclusion is injective.
\item[\reflb{cond:NF1}{(NF)}]  $c_1(\xi)=0$, $c_1(TB)|_{\pi_2(B)}=\lambda\omega|_{\pi_2(B)}$ for some $\lambda>0$ and $\alpha$ is index-admissible.
\end{itemize}
Then $\alpha$ carries at least $r_B$ geometrically distinct contractible periodic orbits, where $r_B:=\dim \H_*(B;\Q)$.
\end{theorem}

\begin{remark}
\label{rmk:general multiplicity1}
The theorem in \cite{GGM2} (namely, \cite[Theorem 2.1]{GGM2}) is actually a bit more general and relaxes the assumption that $\H_{k}(B;\Q)=0$ for every odd $k$ if $c_B>n$. However, in this case the lower bound obtained in \cite{GGM2} is no longer given by the total rank of the homology of $B$. To the best of our knowledge, every prequantization of a symplectic manifold $B$ that admits a contact form with finitely many closed orbits known so far has the property that $\H_*(B;\Q)$ is lacunary. 
\end{remark}

\begin{remark}
\label{rmk:general multiplicity2}
In the theorem in \cite{GGM2}, under the assumption \ref{cond:F1}, the contact form is allowed to be index-negative instead of being index-positive. We do not state the theorem in \cite{GGM2} in its original form because its statement is somewhat technical and involved, and it is not necessary for the purposes of this work.
 \end{remark}

\begin{remark}
The condition that $\Om|_{\pi_2(W)}=0$ in hypothesis \ref{cond:F1} can be dropped and the condition that $c_1(\xi)=0$ in assumption \ref{cond:NF1} can be relaxed to the condition that $c_1(\xi)$ is torsion; see Remark \ref{rmk:hypotheses}.
\end{remark}

Hypothesis \ref{cond:F1} means that $M$ admits a ``nice'' symplectic filling and the assumption \ref{cond:NF1} does not require the existence of a filling at all, with the expense that the contact form has to be index-admissible. This last condition allows us to define the equivariant symplectic homology of $M$ without a filling, using its symplectization; see Section \ref{sec:esh}.

Examples satisfying the hypotheses of the previous theorem include the standard contact sphere $S^{2n+1}$ and the unit cosphere bundle of a compact rank one symmetric space (CROSS) with dimension bigger than two. More precisely, as already mentioned, $S^{2n+1}$ is a prequantization of $\CP^n$, and its obvious filling in $\R^{2n+2}$ clearly satisfies the hypotheses above. A connected Riemannian manifold $N$ is a symmetric space if for every $p \in N$ there exists an isometry $f_p: N \to N$ such that $f_p(p)=p$ and $f_p \circ \exp_p(v)=\exp_p(-v)$ for every $v \in T_pN$. The rank of a symmetric space $N$ is the maximal dimension of a flat totally geodesic submanifold in $N$. By the classification of symmetric spaces, a CROSS is given by one of the following manifolds: $S^m$, $\RP^m$, $\CP^m$, $\HP^m$ and $\CaP^2$; see \cite{Bes} for details. Thus the filling of the unit cosphere bundle $S^*N$ given by the unit codisk bundle $D^*N$ in $T^*N$ meets the condition \ref{cond:F1} of the previous theorem whenever the (real) dimension of $N$ is bigger than two. (In dimension two, we have that $N$ is either $S^2$ or $\RP^2$ which are the only cases where the map $\pi_1(S^*N) \to \pi_1(D^*N)$ is not injective. However, in these cases it is well known that every Reeb flow on $S^*N$ has at least two simple closed orbits.)

Every CROSS $N$ admits a metric such that all of its geodesics are periodic with the same minimal period; in other words, the geodesic flow generates a free circle action on $S^*N$. (To the best of our knowledge, a CROSS is the only known example so far of a closed manifold admitting such a metric \cite{Bes}.) Thus the unit cosphere bundle $S^*N$ is a prequantization of a closed symplectic manifold $(B,\om)$. Moreover, a homological computation shows that $\H_k(B;\Q)=0$ for every odd $k$; see \cite[page 141]{Zil}.  In this case, the total rank $r_B$ of $\H_*(B;\Q)$ and the minimal Chern number $c_B$ are given by Table \ref{tab:prequantizations}.

\begin{table}[tb]
\centering
\begin{tabular}{ | c | c | c | }
\hline
Prequantization & $r_B = \dim \H_*(B;\Q)$ & $c_B$ \\
\hline
\rule{0pt}{0.4cm}
$S^{2n+1}$ & $n+1$ & $n+1$ \\
$S^*S^{2}$ or $S^*\RP^2$ & $2$ & $2$ \\
$S^*S^{m}$ or $S^*\RP^m$ \text{with} $m>2$ \text{even} & $m$ & $m-1$ \\
$S^*S^{m}$ or $S^*\RP^m$ \text{with} $m$ \text{odd} & $m+1$ & $m-1$ \\
$S^*\CP^{m}$ & $m(m+1)$ & $m$ \\
$S^*\HP^m$ & $2m(m+1)$ & $2m+1$ \\
$S^*\CaP^2$ & $24$ & $11$ \\
\hline
\end{tabular}
\vskip .2cm
\caption{Some prequantizations and the corresponding $r_B$ and $c_B$.}
\label{tab:prequantizations}
\end{table}

Hence, we have the following corollary, which was previously proved for the standard contact sphere by Duan, Liu, Long and Wang in \cite{DLLW} and for Finsler metrics on a simply connected CROSS by Duan, Long and Wang in \cite{DLW}.

\begin{corollary}[\cite{GGM2}]
\label{cor:sphere&cross}
Let $(M,\xi)$ be either the standard contact sphere $S^{2n+1}$ or the unit cosphere bundle $S^*N$ of a CROSS and let $\alpha$ be a contact form supporting $\xi$. Assume that $\alpha$ satisfies the conditions of the previous theorem. Then $\alpha$ has at least $r_B$ geometrically distinct periodic orbits, where $r_B$ is given by Table \ref{tab:prequantizations}.
\end{corollary}

Theorem \ref{thm:GGM} is sharp. Indeed, the prequantizations in the previous corollary admit non-degenerate contact forms with precisely $r_B$ geometrically distinct periodic orbits. These contact forms are given by irrational ellipsoids and the Katok-Ziller Finsler metrics \cite{Zil}.

To the best of our knowledge, all the examples of Reeb flows with finitely many closed orbits known so far \emph{are non-degenerate}. When the contact form is non-degenerate, under the index assumptions of Theorem \ref{thm:GGM}, we have at least $r_B$ closed orbits and, as noticed before, this lower bound is sharp. This raises the following hard question:

\vskip .3cm
\noindent {\bf Question:} Let $(M,\xi)$ be a prequantization of a closed symplectic manifold $B$. Is it true that every contact form supporting $\xi$ has at least $r_B$ periodic orbits?
\vskip .2cm

The answer is positive in the very particular case of $M=S^3$ with the standard contact structure. It was proved independently by Cristofaro-Gardiner and Hutchings (in the general case of three dimensional manifolds) \cite{CGH} and Ginzburg, Hein, Hryniewicz and Macarini \cite{GHHM}. In higher dimensions, there are several partial positive results under some hypotheses on the contact form, like index assumptions, convexity and symmetry; see the survey \cite{Mac}.

To the best of our knowledge, all the examples known so far of contact forms with finitely many closed orbits on prequantizations of a symplectic manifold $B$ \emph{have precisely $r_B$ periodic orbits}. This raises the following even harder question:

\vskip .3cm
\noindent {\bf Question:} Let $(M,\xi)$ be a prequantization of a closed symplectic manifold $B$. Is it true that every contact form supporting $\xi$ has either $r_B$ or infinitely many periodic orbits?
\vskip .2cm

The answer is also positive in the case of $M=S^3$ with the standard contact structure as was recently proved by Cristofaro-Gardiner, Hryniewicz, Hutchings and Liu \cite{CGHHL2}. In higher dimensions, it is a widely open and very difficult problem; see however the recent work of Cineli-Ginzburg-G\"urel \cite{CGG} with important results on $S^{2n+1}$ under assumptions of symmetry, convexity and non-degeneracy of the contact form.

A periodic orbit is \emph{elliptic} if every eigenvalue of its linearized Poincar\'e map has modulus one. All the examples of contact forms with finitely many closed orbits that we know so far are non-degenerate and have the property that \emph{every periodic orbit is elliptic}. It is well known that if a non-degenerate periodic orbit $\ga$ is elliptic then $\cz(\ga)=n\ \text{(mod 2)}$. In particular, these contacts forms are \emph{geometrically perfect}, that is, the indexes of every periodic orbit in the same free homotopy class have the same parity. Another example of geometrically perfect contact forms are those with Anosov Reeb flows (e.g. the geodesic flow of a metric with negative sectional curvature) \cite{MP}. In these examples, every periodic orbit is hyperbolic. We do not know examples of geometrically perfect contact forms with elliptic \emph{and} hyperbolic orbits.

We say that a non-degenerate contact form $\alpha$ is \emph{lacunary} if the indexes of every contractible periodic orbit have the same parity. In particular, if every periodic orbit of $\alpha$ is elliptic then $\alpha$ is lacunary. So, all the examples that we know so far of contact forms with finitely many closed orbits are lacunary. Now, we can state the main result of this work, which represents a small step towards the last question. In what follows, a simple contractible closed orbit is a contractible orbit which is not an iterate of another contractible orbit, although it can be the iterate of a non-contractible orbit.

\begin{theorem}
\label{thm:main}
Let $(M^{2n+1},\xi)$ be a prequantization $S^1$-bundle of a closed symplectic manifold $(B,\om)$ such that $\om|_{\pi_2(B)}\neq 0$ and $\H_{k}(B;\Q)=0$ for every odd $k$. Assume that $B$ is spherically positive monotone, that is,  $c_1(TB)|_{\pi_2(B)}=\lambda\omega|_{\pi_2(B)}$ for some $\lambda>0$. Let $\alpha$ be a non-degenerate lacunary contact form supporting $\xi$. Suppose that one of the following two conditions holds:
\begin{itemize}
\item[\reflb{cond:F2}{(F)}] $M$ admits a strong symplectic filling $W$ such that $c_1(TW)|_{\pi_2(W)}=0$ and the map $\pi_1(M) \to \pi_1(W)$ induced by the inclusion is injective.
\item[\reflb{cond:NF2}{(NF)}] $c_1(\xi)|_{H_2(M,\Q)}=0$, $c_B>1$ and $\alpha$ is index-admissible.
\end{itemize}
Then $\alpha$ has precisely $r_B$ simple contractible closed orbits, where $r_B=\dim \H_*(B;\Q)$.
\end{theorem}

We originally proved this theorem under the assumption that $c_B>n/2$ as in Theorem \ref{thm:GGM}. Indeed, the proof given in Section \ref{sec:proof main} assumes this hypothesis. In Appendix \ref{appendix:c_B}, written jointly with Baptiste Serraille, we dropped completely this assumption.

\begin{remark}
Instead of considering only contractible orbits, we can consider \emph{every} periodic orbit. Then the assumption that the map $\pi_1(M) \to \pi_1(W)$ is injective in hypothesis  \ref{cond:F2} can be dropped. The price, of course, is that in this case the condition to be lacunary is stronger, since we ask that the indexes of every closed orbit have the same parity instead of just the indexes of the contractible ones. We consider this situation in \cite{AM5}.
\end{remark}

\begin{remark}
\label{rmk:hypotheses}
The hypothesis \ref{cond:F2} in the previous theorem dropped the assumption that the symplectic form on $W$ is aspherical in Theorem \ref{thm:GGM}. This is possible due to the use of Novikov fields and an action filtration introduced by McLean and Ritter \cite{MR}. Furthermore, the hypothesis \ref{cond:NF2} in the previous theorem is weaker than the assumption \ref{cond:NF1} in Theorem \ref{thm:GGM}: we allow $c_1(\xi)$ to be torsion. See Sections \ref{sec:ESH} and \ref{sec:multiplicity}.
\end{remark}

\begin{remark}
\label{rmk:hyp_cB_NF1}
The assumption $c_B>1$ in \ref{cond:NF2} is used to compute $\HC^0_*(M)$ when we do not have a filling in Proposition \ref{prop:CH}. More precisely, it is used to obtain an index-admissible contact form given by a small perturbation of the connection form; cf. Remark \ref{rmk:hyp_cB_NF2}.
\end{remark}

\begin{remark}
An example of a prequantization satisfying the conditions of Theorem \ref{thm:main} under the assumption \ref{cond:F2} such that the symplectic form of the filling is not aspherical is given by the prequantization $M$ of $S^2 \times S^2 \times S^2$ endowed with the product symplectic form, where the symplectic form on $S^2$ has area 1. One can check that $M$ is simply connected and has a filling $W$, given by a $4$-ball bundle over $S^2 \times S^2$, with vanishing first Chern class. However, the symplectic form on $W$ is not aspherical since the zero section $S^2 \times S^2$ is a symplectic submanifold. For more examples, see Remark \ref{rmk:new examples}.
\end{remark}

\begin{remark}
\label{rmk:c_1 monotone}
We have that $c_1(\xi)|_{H_2(M,\Q)}=0$ if and only if $B$ is monotone, that is,  $c_1(TB)=\lambda[\omega]$ in $\H^2(B;\Q)$ for some $\lambda \in \R$.
\end{remark}

As in Theorem \ref{thm:GGM}, the hypothesis \ref{cond:F2} means that $M$ admits a ``nice'' symplectic filling and the assumption \ref{cond:NF2} does not require the existence of a filling, with the expense that the contact form has to be index-admissible.

We have the following immediate consequences.

\begin{corollary}
Let $(M^{2n+1},\xi)$ be a prequantization $S^1$-bundle of a closed symplectic manifold $(B,\om)$ and $\alpha$ a non-degenerate lacunary contact form supporting $\xi$ as in Theorem \ref{thm:main}. Assume that $\pi_1(M)$ is finite. Then $\alpha$ has precisely $r_B$ simple closed orbits, where $r_B=\dim \H_*(B;\Q)$.
\end{corollary}

Indeed, when $\pi_1(M)$ is finite, the set of simple contractible closed orbits is in bijection with the set of simple closed orbits.

\begin{remark}
One can define a Hamiltonian diffeomorphism of a symplectic manifold $B$ as a \emph{pseudo-rotation} if it has a finite and minimal number of periodic points, although the notion of the minimal number is ambiguous in general \cite{GG18}. Based on our previous discussion, one is tempted to define a \emph{Reeb pseudo-rotation} on a prequantization $M$ as above as a Reeb flow with precisely $r_B$ simple closed orbits. (Another definition, given in \cite{CGGM}, is that the Reeb flow has finitely many simple closed orbits. By our previous discussion, these definitions should be equivalent although it is far from being known.) In this way, the previous corollary says that the Reeb flow of a non-degenerate lacunary contact form on $M$ is a Reeb pseudo-rotation. A natural question is the converse of this result, that is, if a Reeb pseudo-rotation must be the Reeb flow of a non-degenerate lacunary contact form; see Section \ref{sec:questions}. It is true when $M=S^3$ as proved in \cite{CGHHL1}.
\end{remark}

\begin{corollary}
Let $(M^{2n+1},\xi)$ be either the standard contact sphere $S^{2n+1}$ or the unit cosphere bundle $S^*N$ of a CROSS and let $\alpha$ be a non-degenerate lacunary contact form supporting $\xi$. Then $\alpha$ has precisely $r_B$ simple closed orbits, where $r_B$ is given by Table \ref{tab:prequantizations}.
\end{corollary}

The last corollary improves results due to Duan-Liu-Ren \cite{DLR} and Duan-Xie \cite{DX}. In the case of the standard contact sphere $S^{2n+1}$, it follows easily from Theorem \ref{thm:GGM} and \cite[Theorem 1.5]{Gu}. (Actually, we do not need to use Theorem \ref{thm:GGM}: we just have to note that a non-degenerate lacunary contact form on $S^{2n+1}$ is dynamically convex (i.e. every closed orbit has index at least $n+2$) and use the fact that a non-degenerate dynamically convex contact form on $S^{2n+1}$ has at least $n+1$ simple closed orbits \cite{AM2,GKa}. On the other hand, by  \cite[Theorem 1.5]{Gu} the number of simple periodic orbits of a non-degenerate lacunary contact form on $S^{2n+1}$ is at most $n+1$.) As mentioned before, the unit disk bundle $D^*N$ meets the assumption \ref{cond:F2} of Theorem \ref{thm:main} whenever the dimension of $N$ is bigger than two. If $N$ is a surface then it is given by either $S^2$ or $\RP^2$ and in this case the result follows easily from Theorem \ref{thm:GGM} and \cite[Theorem 1.5]{Gu}.

\begin{remark}[More examples]
\label{rmk:new examples}
As explained before, we proved originally Theorem \ref{thm:main} under the assumption that $c_B>n/2$ which is satisfied by the examples of the previous corollary. Without this hypothesis, we have \emph{plenty of other examples} of prequantizations satisfying the assumptions of Theorem \ref{thm:main}. As a matter of fact, let us consider first examples meeting the assumption \ref{cond:F2}. We have, for instance, that  \emph{every} closed toric symplectic manifold $(B,\om)$ satisfies $\om|_{\pi_2(B)}\neq 0$ and $\H_{k}(B;\Q)=0$ for every odd $k$. We claim that if $B$ is monotone and has dimension four then it admits a number $m \in \N$ such that $[\om/m] \in \H^2(B,\Z)$ and the prequantization $(M,\xi)$ of $(B,\om/m)$ has a filling as in hypothesis \ref{cond:F2}. Indeed, take $m \in \N$ such that $[\om/m] \in \H^2(B,\Z)$ and $[\om/m]$ is a primitive class, that is, there is no $c \in \H^2(B,\Z)$ such that $[\om/m]=lc$ for some integer $l\geq 2$. This last condition is equivalent to the existence of an element $S \in \pi_2(B) \cong H_2(B,\Z)$ such that $\lg [\om/m],S \rg=1$. Then one can easily check, using the long exact homotopy sequence of the $S^1$-bundle $M \to B$, that $M$ is simply connected (note that $B$ is simply connected because it is toric). Moreover, since $B$ is monotone, we have that $c_1(\xi)$ is torsion; cf. Remark \ref{rmk:c_1 monotone}. But, since $M$ is simply connected, we conclude from the universal coefficient theorem for cohomology that $\H^2(M,\Z)$ is torsion-free and hence $c_1(\xi)=0$. Thus, $(M,\xi)$ is a simply connected toric contact manifold with $c_1(\xi)=0$. It can be proved that when $M$ has dimension five then it admits a (toric) symplectic filling $W$ such that $c_1(TW)=0$; see \cite[Section 5]{AM3}. (Notice that the symplectic form on $W$ is not aspherical in general.) Hence, $W$ satisfies hypothesis \ref{cond:F2}. Finally, note that if $B$ is monotone then it is positive monotone because it is toric. We can easily construct several examples in higher dimensions; cf. \cite[Section 5]{AM3}. Examples satisfying \ref{cond:NF2} are given by any prequantization of a closed monotone toric symplectic manifold with minimal Chern number bigger than one (e.g. monotone products of complex projective spaces).
\end{remark}

A consequence of the \emph{proof} of Theorem \ref{thm:main} is the following result which provides more examples where our main result holds. Note that if the basis $B$ admits a Hamiltonian action of the torus $T^d$ then we have a lifted contact $T^{d+1}$-action on $M$; see Section \ref{sec:proof main orbifolds}. Given the action of a Lie group $G$ on a manifold $M$ and a vector field $X$ on $M$ invariant under this action, we say that a periodic orbit $\ga$ of $X$ is \emph{symmetric} if $g(\text{Im}(\ga))=\text{Im}(\ga)$ for every $g \in G$.

\begin{theorem}
\label{thm:main orbifolds}
Let $(M^{2n+1},\xi)$ be a prequantization $S^1$-bundle of a closed symplectic manifold $(B,\om)$ such that $\om|_{\pi_2(B)}\neq 0$, $c_B>1$ and $\H_{k}(B;\Q)=0$ for every odd $k$. Assume that $B$ is spherically positive monotone, that is,  $c_1(TB)|_{\pi_2(B)}=\lambda\omega|_{\pi_2(B)}$ for some $\lambda>0$ and that $M$ satisfies the hypothesis \ref{cond:F2} of Theorem \ref{thm:main}. Suppose that $B$ admits a Hamiltonian $T^d$-action with isolated fixed points and consider the corresponding lifted $T^{d+1}$-action on $M$. Let $G$ be a finite subgroup of $T^{d+1}$ acting freely on $M$ and let $\M=M/G$ be the quotient. Then every non-degenerate lacunary contact form $\alpha$ on $\M$ has precisely $r_B$ simple contractible closed orbits, where $r_B=\dim \H_*(B;\Q)$. Moreover, the lifts of these orbits to $M$ are symmetric closed orbits of the lifted contact form $\halpha$ on $M$ with respect to the $G$-action.
\end{theorem}

\begin{remark}
One can show that $\dim \H_*(\M/S^1;\Q)=r_B$, where the $S^1$-action on $\M$ is the one induced by the obvious circle action on $M$, whose orbits are the fibers, which commutes with the $G$-action.
\end{remark}

\begin{remark}
\label{rmk:hyp_cB_NF2}
The hypothesis that $c_B>1$ is used only to ensure that $\alpha$ is index-admissible and therefore we can argue as in Theorem \ref{thm:main} under the assumption \ref{cond:NF2}; cf. Remark \ref{rmk:hyp_cB_NF1}. This hypothesis can be dropped once $\M$ has a filling satisfying \ref{cond:F2}.
\end{remark}

The main point in the previous theorem is that $\M$ is not, in general, a prequantization of a symplectic manifold: it is the prequantization of a symplectic \emph{orbifold}. (Indeed, the $S^1$-action on $\M$ is not free in general: it is only locally free.) For instance, symplectic manifolds meeting the assumptions of the previous theorem are monotone toric closed symplectic manifolds $B$ such that $c_B>1$. Take the particular case where $B=\CP^n$ with the standard symplectic form $\om$ normalized so that $\om$ evaluated at a generator of $\H_2(B,\Z) \cong \Z$ is $\pm 1$. Then $M$ is the standard contact sphere $S^{2n+1}$ which of course has a filling $W$ satisfying the hypotheses of the theorem. Clearly, $\CP^n$ has a Hamiltonian $T^n$-action with isolated fixed points. The lifted $T^{n+1}$-action on $S^{2n+1}$, regarded as a subset of $\C^{n+1}$,  is given by
\begin{equation}
\label{eq:torus action}
(\theta_0,\dots,\theta_n) \cdot (z_0,\dots,z_n) \mapsto (e^{i\theta_0}z_0,\dots,e^{i\theta_n}z_n).
\end{equation}
Given an integer $p\geq 1$, consider the $\Z_p$-action on $S^{2n+1}$ generated by the map
\begin{equation}
\label{eq:Z_p action}
\psi(z_0, \dots, z_n) = \left(e^{\frac{2\pi i \ell_0}{p}}z_0, e^{\frac{2\pi i \ell_1}{p}}z_1, \dots, e^{\frac{2\pi i \ell_n}{p}}z_n \right),
\end{equation}
where $\ell_0, \ldots, \ell_n$ are integers called the weights of the action. Such an action is free when the weights are coprime with $p$ (that we will assume from now on) and in that case we have a lens space obtained as the quotient of $S^{2n+1}$ by the action of $\Z_p$. We denote this lens space by $\L$. In general, $\L$ is not a prequantization of a symplectic manifold: it is a prequantization of a \emph{weighted} complex projective space.

\begin{corollary}
Every non-degenerate lacunary contact form on a lens space $\L$ has precisely $n+1$ closed orbits. Moreover, the lifts of the corresponding contractible closed orbits are symmetric closed orbits on $S^{2n+1}$.
\end{corollary}

Another application of Theorem \ref{thm:main orbifolds} is to unit cosphere bundles of lens spaces, which are also, in general, just prequantizations of orbifolds. Consider the CROSS given by the sphere $S^m$ with the round metric. As mentioned before, the geodesic flow generates a free circle action on $S^*S^m$. Let $B$ be the symplectic manifold given by the quotient $S^*S^m/S^1$. $B$ is the Grassmannian of oriented two-planes $G^+_2(\R^{m+1})$. The linear action of $\SO(m+1)$ on $\R^{m+1}$ naturally induces an action of $\SO(m+1)$ on $B$ which is Hamiltonian. The group $\SO(m+1)$ also induces an isometric action on $S^m$ and it turns out that the action on $B$ is the one induced by the lifted action to $S^*S^m$ (note that the action of $\SO(m+1)$ on $S^*S^m$ sends geodesics to geodesics).

This group has a maximal torus $T$ of dimension $\lfloor \frac{m+1}{2} \rfloor$. Suppose now that $m=2d-1$ is odd so that the dimension of $T$ is $d$. The corresponding action of $T^d$ on $S^m \subset \C^{d}$ is given by rotations in each coordinate as in \eqref{eq:torus action}. The lifted action of $T^{d}$ to $S^*S^m$ coincides with the lift of the Hamiltonian $T^{d}$-action on $B$ to $S^*S^m$ which commutes with the $S^1$-action induced by the geodesic flow, generating the aforementioned contact action of $T^{d+1}$ on $S^*S^m$. Given an integer $p\geq 1$, consider the $\Z_p$-action on $S^{m}$ generated by the map $\psi$ as in \eqref{eq:Z_p action} with the weights $\ell_0, \ldots, \ell_{d-1}$ coprime with $p$. The quotient $S^m/\Z_p$ is the lens space $\Lm$ and the quotient $S^*S^m/\Z_p$ with respect to the lifted action is the unit cosphere bundle of this lens space. Therefore, we conclude the following corollary.

\begin{corollary}
\label{cor:sphere bundle}
Every non-degenerate lacunary contact form on the unit cosphere bundle of a lens space $\Lm$ has precisely $m+1$ closed orbits. Moreover, the lifts of the corresponding contractible closed orbits are symmetric closed orbits on $S^*S^{m}$.
\end{corollary}

\begin{remark}
It follows from the previous discussion and the proof of Theorem \ref{thm:main orbifolds}, presented in Section \ref{sec:proof main orbifolds}, that a lens space admits a Finsler metric with finitely many simple closed orbits (more precisely, $m+1$ closed orbits if the lens space has dimension $m$); see Remark \ref{rmk:finsler}. The point here is that the $\Z_p$-action on $S^*S^m$ is the lift of an action on $S^m$ which is isometric with respect to the Katok-Ziller metric. Thus, a CROSS is not the only example of a closed manifold admitting a Finsler metric with finitely many simple closed orbits; cf. \cite[Page 140]{Zil}.
\end{remark}

\begin{remark}
Using the previous discussion and the proof of Theorem \ref{thm:main orbifolds} one can obtain, inspecting the proof of Theorem \ref{thm:GGM}, multiplicity results for lens spaces and unit cosphere bundles of lens spaces, extending Corollary \ref{cor:sphere&cross}. See Theorems \ref{thm:lens} and \ref{thm:S^*lens} in Section \ref{sec:multiplicity}.
\end{remark}

\begin{remark}
Similarly to Theorem \ref{thm:main}, we originally proved Theorem \ref{thm:main orbifolds} under the assumption that $c_B>n/2$. As in Remark \ref{rmk:new examples}, under the weaker hypothesis that $c_B>1$ we can obtain several new examples where we can apply Theorem \ref{thm:main orbifolds}.
\end{remark}

Theorem \ref{thm:main orbifolds} raises the question if Theorem \ref{thm:main} can be generalized to prequantizations of orbifolds. In \cite{AM5}, we obtain this generalization for several such prequantizations, including a large class of toric contact manifolds. We also prove multiplicity results like Theorem \ref{thm:GGM} for these prequantizations. In particular, we drop the assumption on the minimal Chern number in Theorem \ref{thm:GGM}.

\subsection{Organization of the paper}

The rest of this paper is organized as follows. Sections \ref{sec:ESH} and \ref{sec:IRT} furnish the necessary background for the proof of Theorem \ref{thm:main}. More precisely, in Section \ref{sec:esh} we briefly review several facts about positive equivariant symplectic homology. This homology is computed for suitable prequantization $S^1$-bundles in Section \ref{sec:esh_preq}. The resonance relations for equivariant symplectic homology are presented in Section \ref{sec:resonance}. Section \ref{sec:IRT} states the index recurrence theorem, which is a key tool in the proof of Theorem \ref{thm:main}. The proof of Theorem \ref{thm:main} is established in Section \ref{sec:proof main} under the assumption that $c_B>n/2$. This hypothesis is dropped in Appendix \ref{appendix:c_B}. The proof of Theorem \ref{thm:main orbifolds} is presented in Section \ref{sec:proof main orbifolds}. Section \ref{sec:multiplicity} establishes results about the multiplicity of periodic Reeb orbits on lens spaces and their unit cosphere bundles, extending some results from \cite{GGM2}. Finally, Section \ref{sec:questions} poses some questions concerning contact forms with finitely many closed orbits.

\subsection{Acknowledgments}

We are grateful to Viktor Ginzburg for useful comments on a preliminary version of this work.

\section{Equivariant symplectic homology}
\label{sec:ESH}

\subsection{Equivariant symplectic homology}
\label{sec:esh}
In this section we briefly recall several facts about positive equivariant symplectic homology, treating the subject from a slightly unconventional perspective, following \cite{GGM2}.

Let first $(M,\xi)$ be a closed contact manifold and $(W,\Omega)$ be a strong symplectic filling of $M$ with $c_1(TW)|_{\pi_2(W)}=0$. Usually, we also ask that $\Om|_{\pi_2(W)}=0$ but this condition can be dropped using the universal Novikov field 
\[
\Nov = \bigg\{\sum_{i=1}^\infty n_iT^{a_i}\ ;\, a_i \in \R,\, a_i\to\infty,\, n_i\in \Q\bigg\},
\]
and an action filtration introduced by McLean and Ritter \cite{MR}; cf. \cite[Section 2]{AGKM}. Let $\alpha$ be a non-degenerate contact form on $M$ supporting the contact structure $\xi$. Recall that a periodic orbit $\ga$ of $\alpha$ is \emph{good} if its index has the same parity of the index of the underlying simple closed orbit. Then the positive equivariant symplectic homology $\SH_*^{S^1,+}(W)$ with coefficients in $\Nov$ is the homology of a complex $\CC_*(\alpha)$ generated by the good closed Reeb orbits of $\alpha$; see \cite[Proposition 3.3]{GG}. (More precisely, \cite[Proposition 3.3]{GG} is proved assuming that $\Om|_{\pi_2(W)}=0$ and using $\Q$-coefficients but its proof can be readily adapted to our context since it is purely algebraic; cf. \cite[Section 2]{AGKM}.) This complex is graded by the Conley--Zehnder index and filtered by the action. Furthermore, once we fix a free homotopy class of loops in $W$, the part of $\CC_*(\alpha)$ generated by closed Reeb orbits in that class is a subcomplex. As a consequence, the entire complex $\CC_*(\alpha)$ breaks down into a direct sum of such subcomplexes indexed by free homotopy classes of loops in $W$.

The differential in the complex $\CC_*(\alpha)$, but not its homology, depends on several auxiliary choices, and the nature of the differential is not essential for our purposes. The complex $\CC_*(\alpha)$ is functorial in $\alpha$ in the sense that a symplectic cobordism equipped with a suitable extra structure gives rise to a map of complexes. For the sake of brevity and to emphasize the obvious analogy with contact homology, we denote the homology of $\CC_*(\alpha)$ by $\HC_*(M)$ rather than $\SH_*^{S^1,+}(W)$. The homology of the subcomplex $\CC^0_*(\alpha)$ formed by the orbits contractible in $W$ will be denoted by $\HC_*^0(M)$. However, it is worth keeping in mind that $\CC_*(\alpha)$ and hypothetically even the homology may depend on the choice of the filling $W$.

This description of the positive equivariant symplectic homology as the homology of $\CC_*(\alpha)$ is not quite standard, but it is most suitable for our purposes. (We refer the reader to \cite{GG} for more details and further references and to \cite{BO09,BO10, BO13a, BO13b, BO17,Vit} for the original construction of the equivariant symplectic homology.) To see why $\HC_*(M):=\SH_*^{S^1,+}(W)$ can be obtained as the homology of a single complex generated by good closed Reeb orbits, let us first consider an admissible Hamiltonian $H$ on the symplectic completion of $W$ and focus on the orbits of $H$ with positive action. Such orbits are in a one-to-one correspondence with closed Reeb orbits $\gamma$ with action below a certain threshold $T$ depending on the slope of $H$. The $S^1$-equivariant Floer homology of $H$ is the homology of a Floer-type complex obtained from a non-degenerate parametrized perturbation of $H$; \cite{BO13b,Vit}. This complex is filtered by the action. (Here we are using the action filtration introduced by McLean-Ritter \cite{MR}.) The $E^1$-term of the resulting spectral sequence (over $\Nov$) is generated by the good Reeb orbits of $\alpha$ with action below $T$. Now we can (canonically, once the generators are fixed) reassemble the differentials $\p_r$ into a single differential $\p$ on $\CC_*(H):=E^1_{*,*}$ in such a way the the homology of the resulting complex is $E^\infty=\HF_*^{S^1,+}(H)$. Roughly speaking, $\p=\p_1+\p_2+\ldots$, where $\p_r$ is suitably ``extended'' from $E^r$ to $E^1$. Moreover, this procedure respects the action filtration and is functorial with respect to continuation maps. Passing to the limit in $H$, we obtain the complex $\CC_*(\alpha)$ as the limit of the complexes $\CC_*(H)$; see \cite[Sections 2.5 and 3]{GG} for further details.

A remarkable observation by Bourgeois and Oancea in \cite[Section 4.1.2]{BO17} is that under suitable additional assumptions on the indices of closed Reeb orbits the positive equivariant symplectic homology is defined even when $M$ does not have a symplectic filling. To be more precise, we assume that $c_1(\xi)|_{\pi_2(M)}=0$ and let $\alpha$ be a non-degenerate contact form on $M$ such that all of its closed contractible Reeb orbits have Conley--Zehnder index strictly greater than $3-n$. Furthermore, under this assumption the proof of \cite[Proposition 3.3]{GG} carries over essentially word-for-word, and hence again the positive equivariant symplectic homology of $M$ can be described as the homology of a complex $\CC_*(\alpha)$ generated by good closed Reeb orbits of $\alpha$, graded by the Conley--Zehnder index and filtered by the action. The complex breaks down into the direct sum of subcomplexes indexed by free homotopy classes of loops in $M$. As in the fillable case, we will use the notation $\HC_*(M)$ and $\HC_*^0(M)$. Note that in general, in spite of the notation, this homology has slightly different properties (and hypothetically could be different) from the homology defined via a filling. For instance, it has a decomposition by the free homotopy classes of loops in $M$, but not in $W$ as when $M$ is fillable. (Intuitively, one can think of the resulting homology as defined by using a non-compact filling of $M$ by the bounded part $M\times (0,\,1]$ of the symplectization of $(M,\xi)$.)

\subsection{Equivariant symplectic homology of prequantizations}
\label{sec:esh_preq}

The next proposition, essentially taken from \cite[Proposition 3.1]{GGM2}, shows how to compute the equivariant symplectic homology of a suitable prequantization in terms of the homology of the basis. This computation will be crucial throughout this work.

\begin{proposition}
\label{prop:CH}
Let $(M^{2n+1},\xi)$ be a prequantization of a closed symplectic manifold $(B,\om)$ with $\omega|_{\pi_2(B)} \neq 0$ and  such that $\H_{k}(B;\Q)=0$ for every odd $k$ or $c_B>n$. 
\begin{itemize}
\item[{\rm (a)}] Assume that $M$ satisfies the hypothesis \ref{cond:F2} of Theorem \ref{thm:main}. Then, $B$ is spherically monotone.  When $B$ is spherically positive monotone, the positive equivariant symplectic homology for contractible periodic orbits of $M$ is given by
\begin{equation}
\label{eq:CH}
\HC^0_\ast(M) \cong \bigoplus_{m\in\N} \H_{\ast -2mc_B+n} (B; \Nov). 
\end{equation}
When $B$ is spherically negative monotone, we have
\begin{equation}
\label{eq:CH-nm}
\HC^0_\ast(M) \cong \bigoplus_{m\in\N} \H_{\ast +2mc_B-n} (B; \Nov). 
\end{equation}
In particular, in both cases the homology is independent of the choice of the filling $W$ satisfying the hypothesis \ref{cond:F2} of Theorem \ref{thm:main}.

\item[{\rm (b)}] Alternatively, assume that $B$ is spherically positive monotone with $c_B\geq 2$ and, as in the hypothesis \ref{cond:NF2} of Theorem \ref{thm:main}, $c_1(\xi)|_{H_2(M,\Q)}=0$ and $\alpha$ is a non-degenerate contact form on $(M,\xi)$ such that all contractible closed Reeb orbits have index greater than $3-n$. Then \eqref{eq:CH} also holds.
\end{itemize}
\end{proposition}
In other words, \eqref{eq:CH} asserts that $\HC^0_\ast(M)$ is obtained by taking an infinite number of copies of $H_{\ast-n}(B;\Nov)$ with grading shifted up by positive integer multiples of $2c_B$ and adding up the resulting spaces. We emphasize that in Case (a) of the proposition $\HC^0_\ast(M)$ is the symplectic homology associated with the filling of $M$, while in Case (b) this is the ``non-fillable'' homology described above.

\begin{remark}
In this work, we do not need to consider the case that $B$ is spherically negative monotone, but we included it in Proposition \ref{prop:CH} for the sake of completeness.
\end{remark}

\begin{proof}
First note that, by the universal coefficient theorem,  $\H_*(B;\Nov) \cong \H_*(B;\Q) \otimes \Nov$ since $\Nov$ is a field. Then the proof goes word-for-word the proof of \cite[Proposition 3.1]{GGM2} using the Novikov field and the action filtration mentioned in the previous section, except in the proof of \cite[Lemma 3.3]{GGM2}, used in the proof of item (b), where it is assumed that $c_1(\xi)=0$ while here we are allowing $c_1(\xi)$ to be torsion. This lemma establishes that a sufficiently small non-degenerate perturbation of the connection form is index-admissible, that is, every contractible periodic orbit $\ga$ satisfies $\cz(\ga)>3-n$. The assumption in \cite[Lemma 3.3]{GGM2} that $c_1(\xi)=0$ is used to ensure that the determinant line bundle $\Lambda^n_{\C}\xi$ is trivial. When $c_1(\xi)$ is torsion it is no longer true in general, but we have that $(\Lambda_\C^n\xi)^{\otimes N}$ is a trivial line bundle, where $N$ is the smallest positive integer such that $Nc_1(\xi)=0$. Choose a trivialization $\tau: (\Lambda_\C^n\xi)^{\otimes N} \to M \times \C$ which corresponds to a choice of a non-vanishing section $\sec$ of $(\Lambda_\C^n\xi)^{\otimes N}$. The choice of this trivialization furnishes a unique way to symplectically  trivialize $\oplus_1^N\xi$ along periodic orbits of $\alpha$ up to homotopy. As a matter of fact, given a periodic orbit $\gamma$, let $\Phi: {\gamma^*\oplus_1^N\xi} \to S^1 \times \C^{nN}$  be a trivialization of $\oplus_1^N\xi$ over $\gamma$ as a Hermitian vector bundle such that its highest complex  exterior power coincides with $\tau$. This condition fixes the homotopy class of $\Phi$:  given any other such trivialization $\Psi$ we have, for every $t \in S^1$, that $\Phi_t \circ \Psi_t^{-1}: \C^{nN} \to \C^{nN}$  has complex determinant equal to one and therefore the Maslov index of the symplectic path $t \mapsto \Phi_t \circ \Psi_t^{-1}$ vanishes, where $\Phi_t:=\pi_2 \circ \Phi|_{{\gamma^*\oplus_1^N\xi}(t)}$ and  $\Psi_t:=\pi_2 \circ \Psi|_{{\gamma^*\oplus_1^N\xi}(t)}$ with $\pi_2: S^1 \times \C^{nN} \to \C^{nN}$ being the projection onto the second factor; cf. \cite{AM4,McL2}. Notice that this trivialization is closed under iterations, that is, the trivialization induced on $\gamma^j$ coincides, up to homotopy, with the $j$-th iterate of the trivialization over $\gamma$.

Now, one can define the Conley-Zehnder index $\cz(\gamma;\sec)$ of a closed orbit $\gamma$ in the following way. By the previous discussion, $\sec$ induces a unique up to homotopy symplectic trivialization $\Phi: {\gamma^*\oplus_1^N\xi} \to S^1 \times \R^{2nN}$. Using this trivialization, the linearized Reeb flow gives the symplectic path
\[
\Gamma(t) = \Phi_t \circ \oplus_1^N d\phi_\alpha^t(\gamma(0))|_\xi \circ \Phi_0^{-1},
\]
where $\phi^t_\alpha$ is the Reeb flow of $\alpha$. Then the Conley-Zehnder index is defined as
\[
\cz(\gamma;\sec)=\frac{\cz(\Ga)}{N}.
\]
It turns out that if $\ga$ is contractible then this index is an integer and does not depend on the choice of $\sec$ since the trivialization of the contact structure is homotopic to a trivialization over a capping disk;  see \cite[Section 3]{AM3}.

The previous discussion allows us to define the mean index for all finite segments of Reeb orbits, not necessarily closed; see, e.g., \cite{Es}. This index depends continuously on the initial condition and the contact form (in the $C^2$-topology), and for closed Reeb orbits it agrees with the standard mean index. In this way, it is easy to see that the proof of \cite[Lemma 3.3]{GGM2} works verbatim under the assumption that $c_1(\xi)$ is torsion.
\end{proof}

\subsection{Resonance relations}
\label{sec:resonance}

Let $\ga$ be an isolated (possibly degenerate) closed Reeb orbit and denote by $\HC_*(\ga)$ its local equivariant symplectic homology; see \cite{GG,HM}.  For a non-degenerate orbit $\ga$, we have that
\begin{equation*}
\HC_*(\gamma) =
\begin{cases}
\Nov\ \ \text{if }*=\cz(\ga)\text{ and }\ga\text{ is good} \\
0\ \ \text{otherwise}.
\end{cases}
\end{equation*}
The Euler characteristic of $\ga$ is defined as
\[
\chi(\gamma) = \sum_{m \in \Z} (-1)^m\dim \HC_m(\gamma).
\]
This sum is finite. When $\gamma$ is non-degenerate,
\begin{equation*}
\chi(\gamma) =
\begin{cases}
(-1)^{\mu(\gamma)}\ \ \text{if }\ga\text{ is good} \\
0\ \ \text{otherwise}.
\end{cases}
\end{equation*}
The local {\it mean} Euler characteristic of $\gamma$ is
\[
\hat\chi(\gamma) = \lim_{j\to\infty} \frac{1}{j} \sum_{k=1}^j \chi(\gamma^k).
\]
The limit above exists and is rational; see \cite{GGo}. When $\ga$ is strongly non-degenerate, i.e. all iterates of $\ga$ are non-degenerate, we have
\begin{equation*}
\hat\chi(\gamma) =
\begin{cases}
(-1)^{\cz(\ga)}\ \text{ if }\ga^2\text{ is good} \\
(-1)^{\cz(\ga)}/2\ \text{ if }\ga^2\text{ is bad}.
\end{cases}
\end{equation*}

Assume now that $\alpha$ is index-positive/index-negative and has finitely many distinct simple contractible closed orbits $\ga_1,\dots,\ga_r$. This assumption ensures that the positive/negative mean Euler characteristic
\begin{equation*}
\label{eq:def_MEC}
\chi_\pm(M) := \lim_{j\to\infty} \frac{1}{j} \sum_{m=0}^j (-1)^m b_{\pm m}
\end{equation*}
is well defined, where $b_m := \dim \HC^0_m(M)$ is the $m$-th Betti number; see \cite{GGo}. The mean Euler characteristic is related to local equivariant symplectic homology via the resonance relation
\begin{equation}
\label{eq:resonance}
\sum_{i=1}^{r} \frac{\hat\chi(\gamma_i)}{\mi (\gamma_i)} = \chi_\pm(M),
\end{equation}
proved in \cite{GK} in the non-degenerate case and in \cite{HM} in general. Here the right-hand side is $\chi_+$ when $\alpha$ is index-positive and $\chi_-$ when $\alpha$ is index-negative.

\section{Index recurrence}
\label{sec:IRT}

A crucial ingredient in the proof of Theorem \ref{thm:main} is the following combinatorial result addressing the index behavior under iterations taken from \cite[Theorem 4.1]{GGM2}. This result can also be deduced from the so-called enhanced common index jump theorem due to Duan, Long and Wang \cite{DLW}; see also \cite{Lon02,LZ}.

\begin{theorem}[\cite{GGM2}]
\label{thm:IRT}
Let $\Phi_1,\ldots,\Phi_r$ be a finite collection of strongly non-degenerate elements of $\TSp(2n)$ with $\mi(\Phi_i)>0$ for all $i$. Then for any $\eta>0$ and any $\ell_0\in\N$, there exist two integer sequences $d_j^\pm\to\infty$ and two sequences of integer vectors $\vk^\pm_j=\big(k_{1j}^\pm,\ldots,k_{rj}^\pm\big)$ with all components going to infinity as $j\to\infty$, such that for all $i$ and $j$, and all $\ell\in\Z$ in the range $1\leq |\ell|\leq \ell_0$, we have
\begin{enumerate}
\item[\reflb{cond:i}{\rm{(i)}}]
  $\big|\mi\big(\Phi^{k_{ij}^\pm}_i\big)-d_j^\pm\big|<\eta$ with the
  equality
  $\mi\big(\Phi^{k_{ij}^\pm}_i\big) =
  \cz\big(\Phi^{k_{ij}^\pm}_i\big)=d_j^\pm$ whenever $\Phi_i(1)$ is
  hyperbolic,
\item[\reflb{cond:ii}{\rm{(ii)}}]
  $\cz\big(\Phi^{k_{ij}^\pm+\ell}_i\big)= d_j^\pm + \cz(\Phi^\ell_i)$,
  and
\item[\reflb{cond:iii}{\rm{(iii)}}]
  $\cz\big(\Phi^{k_{ij}^-}_i\big)-d_j^-=
  -\big(\cz\big(\Phi^{k_{ij}^+}_i\big)-d_j^+\big)$.
\end{enumerate}
Furthermore, for any $N\in \N$ we can make all $d_j^\pm$ and $k_{ij}^\pm$ divisible by~$N$.
\end{theorem}

\section{Proof of Theorem \ref{thm:main} under the assumption that $c_B>n/2$}
\label{sec:proof main}

Throughout this section, we will assume that $c_B>n/2$. As explained in the introduction, this hypothesis is dropped in Appendix \ref{appendix:c_B}.

\subsection{Outline of the proof}

First of all, let us give an idea of the proof of Theorem \ref{thm:main} in the case where $n$ is odd. The idea for even $n$ is similar, and the detailed proofs in both cases are in Section \ref{sec:proof main: details}.

The fact that the contact form $\alpha$ is non-degenerate and lacunary implies that every contractible periodic orbit of $\alpha$ is good and the differential in $\CC_*^0(\alpha)$ vanishes, where $\CC_*^0(\alpha)$ is the subcomplex of $\CC_*(\alpha)$ formed by the contractible closed orbits of $\alpha$. Using this, the computation of $\HC^0_*(M)$ given by \eqref{eq:CH} and a combinatorial lemma (Lemma \ref{lemma:combinatorial}) we can show that $\alpha$ has finitely many simple contractible closed orbits. The idea is that, since every contractible orbit contributes to $\HC^0_*(M)$, the rank of $\HC^0_*(M)$ grows with respect to the number of simple contractible orbits, getting a contradiction if we have infinitely many orbits. This part of the proof does not depend on the parity of $n$.

So we have finitely many simple contractible orbits $\{\ga_1,\dots,\ga_r\}$ and we have to show that $r=r_B$. Since the differential in $\CC_*^0(\alpha)$ vanishes, we have, using \eqref{eq:CH}, that every contractible orbit $\ga$ satisfies $\cz(\ga) = n\ \text{(mod 2)}$, $\cz(\ga) \geq \kmin:=\min\{k \in \Z;\ \HC^0_k(M)\neq 0\} \geq 1$ and $\mi(\ga)>0$. Using this and Theorem \ref{thm:IRT}, we find integers $d,k_1,\dots,k_r$ such that $d=2sc_B$ for some $s\in \N$,
\[
\cz(\ga_i^{k_i -\ell}) \leq d - 1\quad\forall 1 \leq \ell < k_i
\]
and
\[
\cz(\ga_i^{k_i +\ell}) \geq d + 1\quad\forall \ell \geq 1.
\]
This implies that
\begin{align*}
\sum_{m=\kmin}^{d} c_m & = \sum_{i=1}^r \#\{1\leq j\leq k_i;\ \cz(\ga_i^j)\leq d\} \\ 
& =  \sum_{i=1}^r k_i -  \sum_{i=1}^r \#\{1\leq j\leq k_i;\ \cz(\ga_i^j)>d\} \\ 
& = sr_B - r_+,
 \end{align*}
where $c_m$ is the $m$-th Morse type number, defined as the number of contractible periodic orbits (simple or not) with index $m$, $r_+:=\#\{1\leq i\leq r;\ \cz(\ga_i^{k_i}) > d\}$ and the equation $\sum_{i=1}^r k_i  = sr_B$ follows from Lemma \ref{lemma:resonance} which is proved using the resonance relations mentioned in Section \ref{sec:resonance}. This equation says that the truncated mean Euler characteristic $\sum_{m=\kmin}^{d} c_m$ (note that $c_m=0$ if $m$ is even) equals $sr_B$ up to the correction term $r_+$.

On the other hand, let $b_m$ be the $m$-th Betti number. We have, by the fact that the differential vanishes and \eqref{eq:CH},
\begin{align*}
\sum_{m=\kmin}^{d} c_m & = \sum_{m=\kmin}^{d} b_m \\
& = sr_B - \sum_{i=0}^{n-1} \dim \H_i(B;\Q) \\
& = sr_B - r_B/2,
\end{align*}
implying that $r_+=r_B/2$. It furnishes at least $r_B/2$ contractible simple orbits. Note that in the second and third equations we used the fact that $n$ is odd and in the second equation we used the hypothesis that $c_B>n/2$.

Repeating the above argument and applying again Theorem \ref{thm:IRT}, we can find integers $d',k'_1,\dots,k'_r$ such that
\[
r_-:=\#\{1\leq i\leq r;\ \cz(\ga_i^{k'_i})>d'\} = r_B/2.
\]
Now the point is that these integers, by property \ref{cond:iii} of Theorem \ref{thm:IRT}, satisfy
\[
\#\{1\leq i\leq r;\ \cz(\ga_i^{k'_i})>d'\} = \#\{1\leq i\leq r;\ \cz(\ga_i^{k_i})<d\},
\]
giving more distinct $r_B/2$ contractible simple orbits.

Finally, to prove that $r=r_B$, we note that
\[
\#\{1\leq i\leq r;\ \cz(\ga_i^{k_i})<d\} + \#\{1\leq i\leq r;\ \cz(\ga_i^{k_i})>d\} = r_B.
\]
But $d=2sc_B$ is even which implies that $\cz(\ga_i^{k_i})\neq d$ since $\cz(\ga_i^{k_i})=n\ \text{(mod 2)}$ is odd. Therefore,
\[
\#\{1\leq i\leq r;\ \cz(\ga_i^{k_i})<d\text{ or }\cz(\ga_i^{k_i})>d\} = r.
\]

\subsection{Proof of the theorem}
\label{sec:proof main: details}

As explained in Section \ref{sec:esh}, under the assumptions of the theorem, we have the positive equivariant symplectic homology of contractible orbits $\HC^0_*(M)$ given by \eqref{eq:CH}. Note that if $M$ does not admit a ``nice'' filling, in the sense of assumption \ref{cond:F2}, we assume that our non-degenerate lacunary contact form $\alpha$ is index-admissible.

By the discussion in the same section, $\HC^0_*(M)$ is the homology of a chain complex $\CC^0_*(\alpha)$ generated by the good contractible closed orbits of $\alpha$. Since $\alpha$ is lacunary, every periodic orbit is good and the differential in $\CC^0_*(\alpha)$ vanishes. Therefore,
\begin{equation}
\label{eq:perorbs}
\HC^0_k(M) \cong \CC^0_k(\alpha) \cong \oplus_{\ga \in \PP^0(\alpha);\, \cz(\ga)=k} \Lambda
\end{equation}
for every $k \in \Z$, where $\PP^0(\alpha)$ is the set of (not necessarily simple) contractible closed orbits of $\alpha$. In other words, every contractible closed orbit of $\alpha$ contributes to the positive equivariant symplectic homology.

It follows from \eqref{eq:perorbs} and \eqref{eq:CH} that, for every contractible orbit $\ga$ of $\alpha$,
\begin{equation}
\label{eq:mi}
\mi(\ga)>0.
\end{equation}
Indeed, since $|j\mi(\ga) - \cz(\ga^j)|<n$ for every $j$, if we have some $\ga$ such that $\mi(\ga) \leq 0$ then we have either that $\HC^0_k(M)$ is non-trivial in arbitrarily large negative degrees $k$ (if  $\mi(\ga) < 0$) or $\HC^0_k(M)$ has infinite dimension in some degree $k \in (-n,n)$  (if $\mi(\ga) = 0$) and both situations are impossible by \eqref{eq:CH}.

Define
\[
\kmin:=\min\{k \in \Z;\ \HC^0_k(M)\neq 0\}.
\]
By \eqref{eq:CH} and our hypotheses (including the condition that $c_B>n/2$),
\[
\kmin=2c_B-n\geq 1.
\]

In what follows, we will prove Theorem \ref{thm:main} under the assumption that $c_B>n/2$. For a clear exposition, we will break down the proof in claims and steps.

\begin{remark}
\label{rmk:key point}
As the reader can easily check, the only point in the proof below that uses that $M$ is a prequantization is the computation of $\HC^0_*(M)$ given by \eqref{eq:CH}. The proof works verbatim for any closed contact manifold $M$ (under the assumption \ref{cond:F2} or that $\alpha$ is index-admissible) whenever we have \eqref{eq:CH} with $c_B>n/2$. (This assumption on $c_B$ can be completely dropped as explained in Appendix \ref{appendix:c_B}.) It will be crucial in the proof of Theorem \ref{thm:main orbifolds} established in Section \ref{sec:proof main orbifolds}.
\end{remark}

\vskip .3cm
\noindent
{\bf Claim 1. $\alpha$ has finitely many simple contractible closed orbits}
\vskip .2cm

This part of the proof follows from the work of G\"urel \cite[Theorem 1.5]{Gu}. For the sake of completeness, we will reproduce her argument. Define
\[
b = \limsup_{k\to\infty} \sum_{i=0}^{2n} \dim \HC^0_{k+i}(M).
\]
By \eqref{eq:CH}, $b$ is a finite integer.

The key point is the following combinatorial lemma proved in \cite{Gu} which uses the aforementioned fact that every contractible orbit of $\alpha$ has positive mean index:

\begin{lemma} \cite[Lemma 3.2]{Gu}
\label{lemma:combinatorial}
Assume that $\alpha$ has a collection of $m$ geometrically distinct contractible periodic orbits $\ga_1,\dotsc, \ga_m$. Then for every sufficiently small $\ep>0$, there exist infinitely many disjoint intervals $I$ of length $2n+\ep$ such that for some positive integers $k_1\geq 1,\dots, k_m \geq 1$ (depending on the interval), the iterated orbits $\ga_1^{k_1},\dotsc, \ga_m^{k_m}$ all have indexes in the interval $I$.
\end{lemma}

Since $\kmin > -\infty$, it follows from \eqref{eq:perorbs} and the previous lemma that $\alpha$ has at most $b$ simple contractible closed orbits.

\begin{remark}
Note that the dimension of $M$ in \cite{Gu} is $2n-1$ and here is $2n+1$.
\end{remark}

\begin{remark}
\label{eq:rmk:bound b}
When $M$ is the standard contact sphere $S^{2n+1}$ we have $b=r_{B}=n+1$. However, there are several examples where $b>r_B$: for instance, when $M=S^*S^m$ we have $r_B=m$ when $m$ is even (resp. $r_B=m+1$ when $m$ is odd) and $b=m+2$ when $m$ is even (resp. $b=m+3$ when $m$ is odd).
\end{remark}

\vskip .3cm
\noindent
{\bf Claim 2. $\alpha$ has precisely $r_B$ simple contractible closed orbits}
\vskip .2cm

This is the more involved part of the proof. In what follows, we will use some ideas from \cite{DX} and \cite{GGM2}. We will also use the aforementioned fact that, by the universal coefficient theorem, $\H_*(B;\Nov) \cong \H_*(B;\Q) \otimes \Nov$ since $\Nov$ is a field, which implies that $\dim \H_k(B;\Nov) = \dim \H_k(B;\Q)$ for every $k$.

By the first claim, we have finitely many simple contractible closed orbits $\{\ga_1,\dots,\ga_r\}$. By \eqref{eq:perorbs}, \eqref{eq:CH} and our hypothesis that $\H_k(B;\Q)=0$ for every odd $k$, we conclude that
\begin{equation}
\label{eq:parity}
\cz(\ga_i^j)=n\ \text{(mod 2)}
\end{equation}
for every $i \in \{1,\dots,r\}$ and $j \in \N$. In particular, every periodic orbit is good. It also follows from \eqref{eq:perorbs} and \eqref{eq:CH} that
\begin{equation}
\label{eq:minind}
\quad\cz(\ga_i^j) \geq \kmin \geq 1\quad\forall i,j.
\end{equation}

By \eqref{eq:mi} we have that
\[
\ell_0:=\max_{1\leq i \leq r}\{\min\{k_0 \in \N;\ \cz(\ga_i^{k+\ell}) \geq  \cz(\ga_i^{k})\ \forall k\geq 1\text{ and }\forall \ell \geq k_0\}\}
\]
is well defined, i.e., the minima are finite (it follows from the aforementioned fact that $|j\mi(\ga_i) - \cz(\ga_i^j)|<n$ for every $j$). By Theorem \ref{thm:IRT}, given $N \in \N$, $\eta>0$ and $\ell_0$ as above we have two sequences of integer vectors $(d^\pm_j,k^\pm_{1j},\dots,k^\pm_{rj})$ satisfying conditions \ref{cond:i}, \ref{cond:ii} and \ref{cond:iii} and such that all $d^\pm_j,k^\pm_{1j},\dots,k^\pm_{rj}$ are divisible by $N$. We will only need one such vector from each sequence. Hence set
\begin{equation}
\label{eq:integers}
(d,k_1,\dots,k_r):=(d^+_1,k^+_{11},\dots,k^+_{r1})\text{ and }(d',k'_1,\dots,k'_r):=(d^-_1,k^-_{11},\dots,k^-_{r1}).
\end{equation}
The following lemma is one of the key steps in the proof and it is proved using the resonance relations mentioned in Section \ref{sec:resonance}; cf.\ \cite[Sublemma 5.2]{AM2} and \cite[Lemma 5.1]{GGM2}.

\begin{lemma}
\label{lemma:resonance}
The numbers $N$ and $\eta$ can be chosen such that $d=2sc_B$ for some $s \in \N$ and
\[
\sum_{i=1}^r k_i = sr_B.
\]
The same holds for $d',k'_1,\dots,k'_r$.
\end{lemma}

\begin{proof}
Let $N$ be any positive integer multiple of $2c_B$ so that $d=2sc_B$ for some $s \in \N$. It is easy to see from \eqref{eq:CH} that
\[
\chi_+(M) = (-1)^n \frac{r_B}{2c_B},
\]
Now, take $\eta$ sufficiently small such that $\eta \sum_{i=1}^r  \frac{1}{\mi(\ga_i)}<1$. Using the resonance relation \eqref{eq:resonance} and \eqref{eq:parity}, we conclude that
\begin{align*}
d\cdot\chi_+(M) 
& = (-1)^n \sum_{i=1}^r \frac{d}{\mi (\gamma_i)} \\
& = (-1)^n \sum_{i=1}^r  k_i +  (-1)^n \sum_{i=1}^r  \frac{(d-k_i\mi(\ga_i))}{\mi(\ga_i)} \\
& = (-1)^n \sum_{i=1}^r  k_i +  (-1)^n \sum_{i=1}^r  \frac{(d-\mi(\ga_i^{k_i}))}{\mi(\ga_i)},
\end{align*}
where in the last equation we used the homogeneity of the mean index, that is, $\mi(\ga_i^j)=j\mi(\ga_i)$ for every $j$. By property \ref{cond:i} of Theorem \ref{thm:IRT} and the condition on
$\eta$,
\[
\bigg| \sum_{i=1}^r \frac{(d-\mi(\ga_i^{k_i}))}{\mi(\ga_i)}\bigg| < \eta \sum_{i=1}^r\frac{1}{\mi(\ga_i)} <1.
\]
Note that by our choice of $N$ the numbers $d\cdot\chi_+(M)$ and $k_i$ for all $i$ are integers. Therefore,
\[
d\cdot\chi_+(M) = (-1)^n \sum_{i=1}^r  k_i.
\]
Obviously, the same argument works for $d',k'_1,\dots,k'_r$.
\end{proof}

Let us now break down the proof of Claim 2 into two cases, according to the parity of~$n$.

\vskip .3cm
\noindent
{\bf Case 1. $n$ is odd}
\vskip .2cm

We will split the proof in this case in three steps.

\vskip .3cm
\noindent
{\bf Step 1. $\alpha$ has at least $r_B/2$ simple contractible closed orbits}
\vskip .2cm

By property \ref{cond:ii} of Theorem \ref{thm:IRT} and \eqref{eq:minind},
\[
\cz(\ga_i^{k_i -\ell}) = d - \cz(\ga_i^\ell) \leq d - 1,\quad\forall 1 \leq \ell \leq \ell_0,
\]
\[
\cz(\ga_i^{k_i +\ell}) = d + \cz(\ga_i^\ell) \geq d + 1,\quad\forall 1 \leq \ell \leq \ell_0,
\]
Choosing $N$ big enough, we can assume that $\ell_0+2\leq \min_{1\leq i\leq r} k_i$. By the definition of $\ell_0$, item \ref{cond:ii} of Theorem \ref{thm:IRT} and \eqref{eq:minind}, we have, for all $\ell_0+1\leq \ell < k_i$,
\[
\cz(\ga_i^{\overbrace{k_i -\ell}^{\geq 1}}) \leq \cz(\ga_i^{\overbrace{k_i -\ell}^{\geq 1}+\overbrace{\ell-1}^{\geq \ell_0}}) = \cz(\ga_i^{k_i -1}) = d - \cz(\ga_i) \leq d - 1
\]
and, for all $\ell \geq \ell_0+1$,
\[
\cz(\ga_i^{k_i +\ell}) = \cz(\ga_i^{\overbrace{k_i +1}^{\geq 1} +\overbrace{\ell-1}^{\geq \ell_0}}) \geq \cz(\ga_i^{k_i +1}) = d + \cz(\ga_i) \geq d + 1.
\]
Thus, we have
\begin{equation}
\label{eq:bound1}
\cz(\ga_i^{k_i -\ell}) \leq d - 1\quad\forall 1 \leq \ell < k_i
\end{equation}
and
\begin{equation}
\label{eq:bound2}
\cz(\ga_i^{k_i +\ell}) \geq d + 1\quad\forall \ell \geq 1.
\end{equation}

Let $c_m=\#\{\ga\in \PP^0(\alpha);\,\cz(\ga)=m\}$ be the $m$-th Morse type number and $b_m=\dim \HC^0_m(M)$ the $m$-th Betti number (which, by \eqref{eq:perorbs}, coincide). Note that $c_m$ and $b_m$ vanish if $m$ is even (recall that $n$ is odd). By \eqref{eq:bound2}, no periodic orbit $\ga_i^{k_i +\ell}$, $\ell\geq 1$, contributes to $\sum_{m=\kmin}^{d} c_m$. Hence,
\begin{align*}
\sum_{m=\kmin}^{d} c_m & = \sum_{i=1}^r \#\{1\leq j\leq k_i;\ \cz(\ga_i^j)\leq d\} \\
& =  \sum_{i=1}^r k_i -  \sum_{i=1}^r \#\{1\leq j\leq k_i;\ \cz(\ga_i^j)>d\} \\
& = sr_B - r_+ \numberthis \label{eq:sumc_m-odd},
 \end{align*}
where $r_+=\#\{1\leq i\leq r;\ \cz(\ga_i^{k_i})>d\}$ and the last equation follows from Lemma \ref{lemma:resonance} and \eqref{eq:bound1}.

On the other hand, by \eqref{eq:CH} and the hypothesis that $c_B>n/2$, we have that
\begin{align*}
\sum_{m=\kmin}^{d} b_m & = sr_B - \sum_{i=0}^{n-1} \dim \H_i(B;\Q) \\
& = sr_B - r_B/2 \numberthis \label{eq:sumb_m-odd},
\end{align*}
where in the first equality we used the fact that $n$ is odd and the last equality holds by Poincar\'e duality and using the facts that $n$ is odd and the homology of $B$ is lacunary.

But, by \eqref{eq:perorbs},
\[
\sum_{m=\kmin}^{d} b_m = \sum_{m=\kmin}^{d} c_m
\]
which implies that
\[
r_+=r_B/2.
\]

\vskip .3cm
\noindent
{\bf Step 2. Existence of other $r_B/2$ simple contractible closed orbits}
\vskip .2cm

Applying the argument of the last step for the integer vector $(d',k_1^-,\dots,k_r^-)$ in \eqref{eq:integers} provided by Theorem \ref{thm:IRT}, we get
\[
r_-:=\#\{1\leq i\leq r;\ \cz(\ga_i^{k'_i})>d'\} = r_B/2.
\]
But, by property \ref{cond:iii} of Theorem \ref{thm:IRT}, we have that
\[
\#\{1\leq i\leq r;\ \cz(\ga_i^{k'_i})>d'\} = \#\{1\leq i\leq r;\ \cz(\ga_i^{k_i})<d\}.
\]

\vskip .3cm
\noindent
{\bf Step 3. Existence of precisely $r_B$ simple contractible closed orbits}
\vskip .2cm

By steps 1 and 2,
\[
\#\{1\leq i\leq r;\ \cz(\ga_i^{k_i})<d\} + \#\{1\leq i\leq r;\ \cz(\ga_i^{k_i})>d\} = r_B.
\]
But $d=2sc_B$ is even which implies that $\cz(\ga_i^{k_i})\neq d$ since $\cz(\ga_i^{k_i})=n\ \text{(mod 2)}$ is odd. Therefore,
\[
\#\{1\leq i\leq r;\ \cz(\ga_i^{k_i})<d\text{ or }\cz(\ga_i^{k_i})>d\} = r.
\]

\vskip .3cm
\noindent
{\bf Case 2. $n$ is even}
\vskip .2cm

We will split the proof in this case in four steps. It is similar to the argument in the case that $n$ is odd although a bit more intricate.

\vskip .3cm
\noindent
{\bf Step 1. Existence of $(r_B-\dim \H_n(B;\Q))/2$ simple contractible closed orbits}
\vskip .2cm

By \eqref{eq:parity} and \eqref{eq:minind} we have that $\cz(\ga_i^j)\geq 2$ for every $i$ and $j$. Arguing as in the case that $n$ is odd, we find integers $d,k_1,\dots,k_r$ such that
\begin{equation}
\label{eq:bound3}
\cz(\ga_i^{k_i -\ell}) \leq d - 2\quad\forall 1 \leq \ell < k_i
\end{equation}
and
\begin{equation}
\label{eq:bound4}
\cz(\ga_i^{k_i +\ell}) \geq d + 2\quad\forall \ell \geq 1.
\end{equation}
Thus, no $\ga_i^{k_i +\ell}$, $\ell \geq 1$, contributes to $\sum_{m=\kmin}^{d + 1} c_m$. Hence,
\begin{align*}
\sum_{m=\kmin}^{d + 1} c_m & = \sum_{i=1}^r \#\{1\leq j\leq k_i;\ \cz(\ga_i^j) \leq d + 1\} \\
& = \sum_{i=1}^r  k_i -  \sum_{i=1}^r \#\{1\leq j\leq k_i;\ \cz(\ga_i^j) > d + 1\} \\
& = sr_B - r_+ \numberthis \label{eq:sumc_m-even},
\end{align*}
where $r_+:=\#\{1\leq i\leq r;\ \cz(\ga_i^{k_i}) > d + 1\}$ and the last equation follows from Lemma \ref{lemma:resonance} and \eqref{eq:bound3}.

On the other hand, by \eqref{eq:CH} and the hypothesis that $c_B>n/2$, we have that
\begin{align*}
\sum_{m=\kmin}^{d + 1} b_m & = sr_B - \sum_{i=0}^{n-2} \dim \H_i(B;\Q) \\
& = sr_B - (r_B-\dim \H_n(B;\Q))/2 \numberthis \label{eq:sumb_m-even},
\end{align*}
where the last equality holds by Poincar\'e duality and using the facts that $n$ is even and the homology of $B$ is lacunary.

But, by \eqref{eq:perorbs},
\[
\sum_{m=\kmin}^{d + 1} b_m = \sum_{m=\kmin}^{d + 1} c_m
\]
which implies that
\[
r_+=(r_B-\dim \H_n(B;\Q))/2.
\]

\vskip .3cm
\noindent
{\bf Step 2. Existence of other $(r_B-\dim \H_n(B;\Q))/2$ simple contractible closed orbits}
\vskip .2cm

Applying the argument of the last step for the  integer vector $(d',k_1^-,\dots,k_r^-)$ in \eqref{eq:integers} provided by Theorem \ref{thm:IRT}, we get
\[
r_-:=\#\{1\leq i\leq r;\ \cz(\ga_i^{k'_i})>d' + 1\} = (r_B-\dim \H_i(B;\Q))/2.
\]
But, by property \ref{cond:iii} of Theorem \ref{thm:IRT}, we have that
\[
\#\{1\leq i\leq r;\ \cz(\ga_i^{k'_i})>d' + 1\} = \#\{1\leq i\leq r;\ \cz(\ga_i^{k_i})<d - 1\}.
\]

\vskip .3cm
\noindent
{\bf Step 3. Existence of more $\dim \H_n(B;\Q)$ simple contractible closed orbits}
\vskip .2cm

By \eqref{eq:bound3} and \eqref{eq:bound4}, the only iterate that can contribute to $\HC^0_{d}(M)$ is $\ga_i^{k_i}$. But all the orbits $\ga_i$ obtained in steps 1 and 2 satisfy either $\cz(\ga_i^{k_i})>d + 1$ or $\cz(\ga_i^{k_i})<d - 1$. Thus, we need at least $\dim \HC^0_{d}(M)=\dim \H_n(B;\Q)$ new simple contractible closed orbits (the equality $\dim \HC^0_{d}(M)=\dim \H_n(B;\Q)$ holds by \eqref{eq:CH} and the hypothesis that $c_B>n/2$).

\vskip .3cm
\noindent
{\bf Step 4. Existence of precisely $r_B$ simple contractible closed orbits}
\vskip .2cm

By step 3,
\[
\#\{1\leq i\leq r;\ \cz(\ga_i^{k_i})=d\} = \dim \H_n(B;\Q).
\]
By steps 1 and 2,
\[
\#\{1\leq i\leq r;\ \cz(\ga_i^{k_i})<d - 1\} + \#\{1\leq i\leq r;\ \cz(\ga_i^{k_i})>d + 1\} = r_B -  \dim \H_n(B;\Q). 
\]
But notice that $d$ is even and $d \pm 1$  are odd. Since $\cz(\ga_i^{k_i})$ is even, it cannot be equal to $d \pm 1$. Therefore,
\[
\#\{1\leq i\leq r;\ \cz(\ga_i^{k_i})<d - 1\text{ or }\cz(\ga_i^{k_i})>d + 1\text{ or }\cz(\ga_i^{k_i})=d\} = r.
\]

\section{Proof of Theorem \ref{thm:main orbifolds}}
\label{sec:proof main orbifolds}

Let us first prove that $\alpha$ has precisely $r_B$ contractible closed orbits. First of all, note that $\M$ does not need to have a symplectic filling. Thus, we will argue as in the proof of Theorem \ref{thm:main} under the assumption \ref{cond:NF2}, but for this we have to show that $\alpha$ is index-admissible. Let $\halpha$ be the lift of $\alpha$ to $M$.

\begin{lemma}
\label{lemma:bijection}
Let $\PP^0(\halpha)$ and $\PP^0(\alpha)$ be the set of (not necessarily simple) contractible closed orbits of $\halpha$ and $\alpha$ respectively. Then there exists a map $\psi: \PP^0(\halpha) \to \PP^0(\alpha)$ such that $\cz(\hga)=\cz(\psi(\hga))$. Moreover, this map has an inverse on the right $\rho: \PP^0(\alpha) \to \PP^0(\halpha)$.
\end{lemma}

\begin{proof}
Let $\hga$ be a contractible closed orbit of $\halpha$ and $\tau: M \to \M$ be the quotient projection. Define $\psi(\hga)=\tau\circ\hga$. Clearly, this is a contractible closed orbit of $\alpha$ with the same index. To construct $\rho$, let $\ga$ be a contractible closed orbit of $\alpha$ and choose a point $x_0 \in \tau^{-1}(\ga(0))$. Since $\tau: M \to \M$ is a finite covering and $\ga$ is contractible, we have that $\ga$ admits a closed lift $\hga$ such that $\hga(0)=x_0$. This lift is also contractible because given a capping disk $f$ of $\ga$ (i.e. a continuous map $f: D^2 \to \M$ such that $f|_{\partial D^2}=\ga$) we have that $f$ admits a lift $\hat f$ to $M$ such that $\hat f|_{\partial D^2}=\hga$. Define $\rho(\ga)=\hga$. Then $\psi(\rho(\ga))=\psi(\hga)=\ga$ as desired.
\end{proof}

It follows from the first statement of the previous lemma that $\halpha$ is lacunary. Since $M$ has a filling satisfying the assumption \ref{cond:F2} of Theorem \ref{thm:main}, we can consider the positive equivariant symplectic homology of $M$. As explained in Section \ref{sec:esh}, this is the homology of the complex $\CC^0_*(\halpha)$ generated by the contractible good orbits of $\halpha$. Since $\halpha$ is lacunary, every periodic orbit is good and the differential of $\CC^0_*(\halpha)$ vanishes. Therefore,
\[
\cz(\hga) \geq \kmin=\min\{k \in \Z;\ \HC^0_k(M)\neq 0\}
\]
for every contractible closed orbit $\hga$ of $\halpha$. But, by \eqref{eq:CH},
\[
\HC^0_\ast(M) \cong \bigoplus_{m\in\N} \H_{\ast -2mc_B+n} (B; \Nov)
\]
which implies that $\kmin=2c_B-n$. By hypothesis, $c_B\geq 2$ and therefore  $\kmin>3-n$. Hence, $\halpha$ is index-admissible and so is $\alpha$, since the map $\psi$ in Lemma \ref{lemma:bijection} is surjective.

Thus, we can argue as in Theorem \ref{thm:main} under the assumption \ref{cond:NF2}. As mentioned in Remark \ref{rmk:key point}, the assumption that $M$ is a prequantization in Theorem \ref{thm:main} is needed only to achieve the isomorphism \eqref{eq:CH}. Thus, it is enough to show that
\begin{equation}
\label{eq:CH-orbifold}
\HC^0_\ast(\M) \cong \bigoplus_{m\in\N} \H_{\ast -2mc_B+n} (B; \Nov).
 \end{equation}

To prove \eqref{eq:CH-orbifold}, we will construct a non-degenerate, index-admissible and lacunary contact form $\eta$ on $\M$ such that
\begin{equation}
\label{eq:complex eta}
\CC^0_*(\eta) \cong \bigoplus_{m\in\N} \H_{\ast -2mc_B+n} (B; \Nov)
\end{equation}
which readily implies \eqref{eq:CH-orbifold}.

\begin{remark}
Let $\bxi=\tau_*\xi$ be the contact structure on $\M$, where $\tau: M \to \M$, as before, is the quotient projection. To define the equivariant symplectic homology of contractible orbits $\HC^0_*(\M)$ using the symplectization of $\M$ (with an integral grading) it is enough to have that $c_1(\bxi)|_{\pi_2(\M)}=0$: we do not need that $c_1(\bxi)|_{H_2(\M;\Q)}=0$ as in assumption \ref{cond:NF2}. Indeed, this assumption is needed in Theorem \ref{thm:main} in order to compute $\HC^0_*(\M)$ in item (b) of Proposition \ref{prop:CH}. Here, we do not need this since we compute $\HC^0_*(\M)$ directly from the construction of $\eta$. The condition $c_1(\bxi)|_{\pi_2(\M)}=0$ holds because, by the assumption that $(M,\xi)$ satisfies \ref{cond:F2}, we have that $c_1(\xi)|_{\pi_2(M)}=0$ which is equivalent to the condition $c_1(\bxi)|_{\pi_2(\M)}=0$ since the induced map $\tau_\#: \pi_2(M) \to \pi_2(\M)$ is an isomorphism.
\end{remark}

Consider the Hamiltonian $T^d$-action on $B$. Let $\{X_1,\dots,X_d\}$ be a basis of the Lie algebra of $T^d$ such that the corresponding 1-parameter subgroups are circles. Denote by $\vr^i_t$ the action on $B$ of the circle corresponding to $X_i$. Without fear of ambiguity, we will denote by $X_i$ the Hamiltonian vector field of $\vr^i_t$ and let $H_i: B \to \R$ be the Hamiltonian function that we will assume, without loss of generality, to be a positive function.

\begin{lemma}
There exists a compact 1-parameter subgroup of $T^d$ whose corresponding circle action on $B$ has the property that its fixed point set coincides with the fixed point set of the $T^d$-action.
\end{lemma}

\begin{proof}
Since $B$ is closed, this action has finitely many isotropy subgroups $G_1,\dots,G_m$. At most one of these subgroups has dimension $d$, say $G_m$. Since the codimension of all other stabilizers $G_1,\dots,G_{m-1}$ is positive, we can find a compact 1-parameter subgroup of $T^d$ which is not contained in any $G_i$ with $1\leq i <m$. The existence of this 1-parameter subgroup follows from the fact the union of the compact 1-parameter subgroups of $T^d$ is a dense subset of $T^d$. This is the desired subgroup.
\end{proof}

Thus, we can assume, without loss of generality, that $X_1$ generates a circle action whose fixed points coincide with the fixed points of the $T^d$-action. Let $\beta$ be the connection form on $M$ and consider the contact forms
\[
\eta_i = \beta/\hH_i
\]
where $\hH_i = H_i \circ \pi$ with $\pi: M \to B$ being the quotient projection. We have that the Reeb vector field of $\eta_i$ is given by
\[
R_i = \hH_i R_\beta + X_i^h
\]
where $R_\beta$ is the Reeb flow of $\beta$ (that generates the circle action of the prequantization $M$ whose orbits are the fibers) and $X_i^h$ is the horizontal lift of $X_i$; see \cite[Lemma 3.4]{AGZ}. The flow of $R_i$ generates a circle action \cite[Proposition 3.6]{AGZ} and commutes with the flow of $R_\beta$ \cite[Lemma 3.8]{AGZ}. We also have that $[R_i,R_j]=0$ for every $i$ and $j$. As a matter of fact,
\begin{align*}
[R_i,R_j] & = [\hH_i R_\beta + X_i^h,\hH_j R_\beta + X_j^h] \\
& = [\hH_i R_\beta,\hH_j R_\beta] + [\hH_i R_\beta,X_j^h] + [X_i^h,\hH_j R_\beta] + [X_i^h,X_j^h] \\
& = [X_i^h,X_j^h] \\
& = [X_i,X_j]^h + \om(X_i,X_j)R_\beta \\
& = 0,
\end{align*}
where $[\hH_i R_\beta,\hH_j R_\beta]=0$ because $\hH_i$ and $\hH_j$ are invariant by the flow of $R_\beta$, $[\hH_i R_\beta,X_j^h]=0$ (and similarly $[X_i^h,\hH_j R_\beta]=0$) since $[R_\beta,X_j^h]=0$ (see \cite[Lemma 3.8]{AGZ}) and $X_j^h(\hH_i)=X_j(H_i)=\{H_j,H_i\}=0$, $[X_i,X_j]^h=0$ is the horizontal lift of $[X_i,X_j]$, and $\om(X_i,X_j)=0$ because $X_i$ and $X_j$ commute.

Thus, $R_1,\dots,R_{d+1}$ generate a contact $T^{d+1}$-action on $M$, where $R_{d+1}:=R_\beta$. We claim that this action preserves $\eta_i$ for every $i$. Indeed, let $g \in T^{d+1}$. Since this action preserves the contact structure $\xi=\ker\beta$ we have that $g^*\eta_i=f\eta_i$ for some function $f$. To conclude that $f\equiv 1$ it is enough to show that $g^*\eta_i(R_i)=1$. But $g_*(R_i)=R_i$ because $g=\phi^1_{t_1} \circ \dots \circ \phi^{d+1}_{t_{d+1}}$ for some $(t_1,\dots,t_{d+1}) \in T^{d+1}$, where $\phi^j_{t}$ is the flow of $R_j$, and $d\phi^j_{t}(R_i)=R_i$ for every $j$ and $t$ since $R_1,\dots,R_{d+1}$ commute.

Now, let $\ep>0$ be an irrational number and define $\eta_1^\ep=\beta/(\hH_1+\ep)$ so that the Reeb vector field of $\eta_1^\ep$ is $R_1^\ep=R_1 + \ep R_{d+1}$. By the previous discussion, we have that the $T^{d+1}$-action preserves $\eta^\ep_1$ as well: given $g \in T^{d+1}$ we have that $g^*\eta^\ep_1=f \eta^\ep_1$ for some function $f$ and  $g_*(R_1^\ep)=R_1^\ep$ which implies that $f \equiv 1$. Consequently, $\eta^\ep_1$ induces a contact form $\eta$ on $\M$. We claim that $\eta$ is the desired contact form satisfying \eqref{eq:complex eta}.

As a matter of fact, note that the contractible periodic orbits of $\eta^\ep_1$ and $\eta$ are precisely the contractible iterations of the fibers over the fixed points of the torus action on $B$ (see \cite[Proposition 3.9]{AGZ}) and that all these orbits are elliptic. Indeed, note that the contractible orbits of $\eta^\ep_1$ projects to contractible orbits of $\eta$ and, since $M \to \M$ is a covering, every contractible orbit of $\eta$ is a projection a contractible orbit of $\eta^\ep_1$ (cf. Lemma \ref{lemma:bijection}); therefore, the contractible orbits of $\eta$ are elliptic if, and only if, so are the contractible orbits of $\eta^\ep_1$. But, since $R_1^\ep = (\hH_1+\ep)R_\beta + X_1^h$ and the vector fields $(\hH_1+\ep)R_\beta$  and $X_1^h$ commute, it is easy to see that every contractible orbit of $R_1^\ep$ is elliptic because every singularity of $X_1$ is elliptic. Therefore, $\eta^\ep_1$ and $\eta$ are lacunary. It is also easy to see that, choosing $\ep$ properly, they are also non-degenerate (because one can choose $\ep$ such that, for every fixed point $p$ of the $S^1$-action generated by $X_1$, all eigenvalues of $d\phi_p^\ep(p)$ are of the form $e^{2\pi i\theta}$ with $\theta \in \R\setminus\Q$, where $\phi_p^\ep$ is the time $1/(H_1(p)+\ep)$ map of the flow of $X_1$). Moreover, we have that the map $\psi: \PP^0(\eta^\ep_1) \to \PP^0(\eta)$ furnished by Lemma \ref{lemma:bijection} is a bijection. Indeed, since the contractible periodic orbits of $\eta^\ep_1$ are precisely the contractible iterations of the fibers over the fixed points of the torus action on $B$, we have that the images of theses orbits coincide with the corresponding images of the orbits of the $T^{d+1}$-action and consequently the closed orbits of $\eta^\ep_1$ are symmetric with respect to the $G$-action. It easily follows from this that $\psi$ is injective. Therefore,
\[
\CC^0_*(\eta^\ep_1) \cong \CC^0_*(\eta).
\]
But, since $\eta^\ep_1$ is lacunary, the differential in $\CC^0_*(\eta^\ep_1)$ vanishes and consequently, by \eqref{eq:CH},
\begin{equation}
\CC^0_*(\eta^\ep_1) \cong \HC^0_*(M) \cong \bigoplus_{m\in\N} H_{\ast -2mc_B+n} (B; \Nov),
\end{equation}
proving \eqref{eq:complex eta}.

\begin{remark}
\label{rmk:finsler}
When $M$ is the unit cosphere bundle of a CROSS $N$ with the flow of $R_\beta$ being the (periodic) geodesic flow, the contact form $\eta^\ep_1$ can be chosen such that it is induced by a Finsler metric. In this way, we get the Katok-Ziller Finsler metrics \cite{Zil}. When the action of $G$ on $S^*N$ is the lift of a free $G$-action on $N$ then $\eta$ is a contact form on $S^*(N/G)$ also induced by a Finsler metric on $N/G$.
\end{remark}

Finally, let us prove that the lifts of the contractible closed orbits of $\alpha$ to $M$ are symmetric. Let $\{\ga_1,\dots,\ga_{r_B}\}$ be the set of simple contractible orbits of $\alpha$ and $\{\hga_1,\dots,\hga_{r_B}\}$ be lifts of these orbits, which are contractible closed orbits of the lifted contact form $\halpha$ on $M$. These lifts are also simple contractible orbits. By Theorem \ref{thm:main}, $\halpha$ has precisely $r_B$ simple contractible closed orbits. If one of the orbits $\hga_i$ is not symmetric,  we would have some $g \in G$ such that $\Im(\hga_i) \neq \Im(g(\hga_i))$ and consequently we would have more than $r_B$ simple contractible closed orbits of $\halpha$, a contradiction.

\section{Multiplicity of periodic orbits for lens spaces and their unit cosphere bundles}
\label{sec:multiplicity}

As in Theorem \ref{thm:main} (see Remark \ref{rmk:key point}), the only point in the proof of Theorem \ref{thm:GGM} in \cite{GGM2} that uses that $M$ is a prequantization is to achieve the isomorphism \eqref{eq:CH} (when $B$ is spherically positive monotone). As proved in the previous section, this isomorphism holds for some prequantizations of orbifolds as in Theorem \ref{thm:main orbifolds}. In particular, as explained in the introduction, this isomorphism holds for lens spaces and their unit cosphere bundles. Therefore, we can conclude the following results.

\begin{theorem}
\label{thm:lens}
Let $\alpha$ be a non-degenerate contact form on $\L$ which is index-positive, index-admissible and has no contractible periodic orbits $\ga$ such that $\cz(\ga)=0$ if $n$ is odd or $\cz(\ga) \in \{0,\pm 1\}$ if $n$ is even. Then $\alpha$ carries at least $n+1$ geometrically distinct contractible periodic orbits.
\end{theorem}

\begin{theorem}
\label{thm:S^*lens}
Let $\alpha$ be a non-degenerate contact form on $S^*\Lm$ which is index-positive and has no contractible periodic orbits $\ga$ such that  $\cz(\ga) \in \{0,\pm 1\}$. Then $\alpha$ carries at least $m+1$ geometrically distinct contractible periodic orbits.
\end{theorem}

Note that in Theorem \ref{thm:lens} we assume that $\alpha$ is index-admissible because a lens space does not admit, in general, a filling satisfying the assumption \ref{cond:F2} of Theorem \ref{thm:main}. This hypothesis is not necessary in Theorem \ref{thm:S^*lens} because $S^*\Lm$ satisfies this assumption. Moreover, note that the dimension of $S^*\Lm$ is $2m-1=2(m-1)+1$ with $m-1$ even since $m$ is odd.

\begin{remark}
Using symplectic homology of contact manifolds with orbifold fillings, developed in \cite{GZ}, it is possible that the assumption in Theorem \ref{thm:lens} that $\alpha$ is index-admissible can be dropped.
\end{remark}

\section{Final questions}
\label{sec:questions}

To finish this work, let us pose some final questions. To the best of our knowledge, all the examples known so far of contact manifolds admitting a contact form with finitely many simple periodic orbits are prequantizations of orbifolds.

\vskip .3cm
\noindent {\bf Question:} Is there an example of a contact manifold admitting a contact form with finitely many simple periodic orbits that is not a prequantization of an orbifold?
\vskip .2cm

The answer is negative in dimension three if we assume that the first Chern class of the contact structure is torsion \cite{CGHHL1,CGHHL2}. 

As mentioned in the introduction, all the examples that we know so far of contact forms with finitely many closed orbits are non-degenerate and have only elliptic closed orbits. This raises the following natural question.

\vskip .3cm
\noindent {\bf Question:} Is there an example of a contact form with finitely many closed orbits possessing non-elliptic orbits?
\vskip .2cm

The answer is also negative in dimension three if we assume that the first Chern class of the contact structure is torsion \cite{CGHHL1,CGHHL2}.

Another natural problem, related to the previous question, that arises from our results is the following:

\vskip .3cm
\noindent {\bf Question:} Is it true that a contact form with finitely many closed orbits is non-degenerate and lacunary?
\vskip .2cm

If it is true, we would conclude that every contact form on a prequantization $M$ with finite fundamental group satisfying the assumptions of Theorem \ref{thm:main} and admitting a nice filling, in the sense of assumption \ref{cond:F2}, has either $r_B$ or infinitely many closed orbits. It is true in dimension three if we assume that the first Chern class of the contact structure is torsion \cite{CGHHL1,CGHHL2}.

\appendix
\section{Dropping the assumption on the minimal Chern number\\
(joint with Baptiste Serraille)}
\label{appendix:c_B}

In what follows, we will show how to change the proof of Theorem \ref{thm:main} in order to drop the assumption $c_B>n/2$ used in Section \ref{sec:proof main}. As before, we split the argument in two cases, according to the parity of $n$.

The main point is the reestablishment of equations \eqref{eq:sumc_m-odd}, \eqref{eq:sumc_m-even}, \eqref{eq:sumb_m-odd} and \eqref{eq:sumb_m-even} without the assumption on $c_B$. The outcomes are  \eqref{eq:sumc_m-odd-2}, \eqref{eq:sumc_m-even-2}, \eqref{eq:sumb_m-odd-2} and \eqref{eq:sumb_m-even-2} respectively. (Note that the pairs of equations \eqref{eq:sumc_m-even}, \eqref{eq:sumb_m-even} and \eqref{eq:sumc_m-even-2}, \eqref{eq:sumb_m-even-2} are equivalent under the assumption $c_B>n/2$ because, in this case, by \eqref{eq:CH}, $b_0=0$.) For the first two equations, we have to deal with closed orbits with non-positive index, since without this assumption, it is not true that $\kmin\geq 1$ anymore; see \eqref{eq:minind}. This can be fixed redefining the number $\ell_0$ as
\[
\ell_0:=\max_{1\leq i \leq r}\{\min\{k_0 \in \N;\ \cz(\ga_i^{k+\ell}) \geq  \cz(\ga_i^{k}) + n + 3\ \forall k\geq 1\text{ and }\forall \ell \geq k_0\}\},
\]
which is well defined (i.e. is a finite number), as in Section \ref{sec:proof main}, by \eqref{eq:mi}. The last two equations are somewhat more involved, and can be dealt with using Lemma \ref{lemma:sumb_m} below.

\vskip .3cm
\noindent
{\bf Case 1. $n$ is odd}
\vskip .2cm

\vskip .3cm
\noindent
{\bf Step 1. $\alpha$ has at least $r_B/2$ simple contractible closed orbits}
\vskip .2cm

Choose $N$ in Theorem \ref{thm:IRT} such that $\ell_0+2 \leq \min_{1\leq i \leq r} k_i$. By property \ref{cond:ii} of Theorem \ref{thm:IRT},
\[
\cz(\ga_i^{k_i -\ell}) = d - \cz(\ga_i^\ell)
\]
\[
\cz(\ga_i^{k_i +\ell}) = d + \cz(\ga_i^\ell)
\]
for every $1 \leq \ell \leq \ell_0$. Therefore,
\begin{equation}
\label{eq:1}
\#\{\ell \in [-\ell_0,\ell_0]\setminus\{0\};\, \cz(\ga_i^{k_i +\ell}) \leq d-1\} = \#\{\ell \in [-\ell_0,\ell_0]\setminus\{0\};\, \cz(\ga_i^{k_i +\ell}) \geq d+1\}.
\end{equation}
(Note here that $k_i > \ell_0$.) Since there is no contractible closed orbit with index zero (because $n$ is odd), we get that each index $\cz(\ga_i^\ell)$ is either positive or negative and consequently,
\begin{equation}
\label{eq:2}
\#\{\ell \in [-\ell_0,\ell_0]\setminus\{0\};\, \cz(\ga_i^{k_i +\ell}) \leq d-1\} + \#\{\ell \in [-\ell_0,\ell_0]\setminus\{0\};\, \cz(\ga_i^{k_i +\ell}) \geq d+1\} = 2\ell_0.
\end{equation}
By the definition of $\ell_0$ and the fact that $\cz(\ga) < \mi(\ga)+n$ for every non-degenerate closed orbit $\ga$,
\begin{equation}
\label{eq:3}
\cz(\ga_i^{k_i-\ell}) \leq \cz(\ga_i^{k_i}) - (n+3) \leq \mi(\ga_i^{k_i}) + n - (n+3) \leq d+1 + n - (n+3) = d-2,
\end{equation}
for every $\ell_0+1 \leq \ell < k_i$, where the third inequality holds by property \ref{cond:i} of Theorem \ref{thm:IRT}. Similarly,
\begin{equation}
\label{eq:4}
\cz(\ga_i^{k_i+\ell}) \geq d+2,
\end{equation}
for all $\ell \geq \ell_0$.

Thus, we arrive at
\begin{align*}
\#\{j \in \N,\, j\neq k_i;\, \cz(\ga_i^j) \leq d\} & = \#\{1\leq j \leq k_i+\ell_0,\, j\neq k_i;\, \cz(\ga_i^j) \leq d\} \\
& = k_i - \ell_0 - 1 + \#\{k_i - \ell_0 \leq j \leq k_i+\ell_0,\, j\neq k_i;\, \cz(\ga_i^j) \leq d\} \\
& = k_i - \ell_0 - 1 + \frac12 \#\{k_i - \ell_0 \leq j \leq k_i+\ell_0,\, j\neq k_i\} \\
& = k_i - \ell_0 - 1 + \frac12 2\ell_0 \\
& = k_i - 1 \numberthis \label{eq:k_i},
\end{align*}
where the first equation follows \eqref{eq:4}, the second equation follows from \eqref{eq:3} and the third identity is a consequence of \eqref{eq:1} and \eqref{eq:2}. Hence, we conclude that
\begin{align*}
\sum_{m=\kmin}^{d} c_m & = \sum_{i=1}^r \#\{j \in \N;\ \cz(\ga_i^j)\leq d\} \\
& = \sum_{i=1}^r \#\{j \in \N,\, j\neq k_i;\, \cz(\ga_i^j) \leq d\} + \#\{1\leq i \leq r;\, \cz(\ga_i^{k_i})\leq d\} \\
& = \sum_{i=1}^r k_i  - r + (r-r_+) \\
& = \sum_{i=1}^r k_i - r_+,
\end{align*}
where, as in Section \ref{sec:proof main}, $r_+=\#\{1\leq i\leq r;\ \cz(\ga_i^{k_i})>d\}$ . Therefore, by Lemma \ref{lemma:resonance},
\begin{equation}
\label{eq:sumc_m-odd-2}
\sum_{m=\kmin}^{d} c_m = sr_B - r_+,
\end{equation}
as in \eqref{eq:sumc_m-odd}. (Note that we do not use any assumption on $c_B$ in Lemma \ref{lemma:resonance}.)

Now, we claim that
\begin{equation}
\label{eq:sumb_m-odd-2}
\sum_{m=\kmin}^{d} b_m = sr_B - r_B/2
\end{equation}
as in \eqref{eq:sumb_m-odd}. Assuming the claim, this first step follows from \eqref{eq:sumc_m-odd-2} and \eqref{eq:sumb_m-odd-2} just as in Section \ref{sec:proof main}. {\bf Step 2} and {\bf Step 3} follow as in the rest of the proof.

So to finish the argument when $n$ is odd it remains only to prove \eqref{eq:sumb_m-odd-2}. It is an immediate consequence of the next lemma (note that $n$ is odd and $\H_*(B;\Q)$ vanishes in odd degrees). Recall that $d=2sc_B$.

\begin{lemma}
\label{lemma:sumb_m}
We can choose $s$ arbitrarily large such that
\[
2\sum_{m=\kmin}^{d} b_m = (2s-1)r_B + \dim \H_n(B;\Q) + 2\sum_{m=s+1}^{2s-1} \dim \H_{n+2(s-m)c_B}(B;\Q).
\]
The result holds for any $n$ (odd or even).
\end{lemma}

\begin{proof}
Take $N$ big enough in Theorem \ref{thm:IRT} such that $2sc_B > 2n$. Change the grading of $\HC_*(M)$ to $*' = 4sc_B - *$ so that
\[
\HC_{*'}(M) \cong \bigoplus_{m \in \N} \H_{4sc_B-*'-2mc_B+n}(B;\Lambda).
\]
Thus, since $\kmin=2c_B-n$,
\begin{align*}
2\sum_{m=\kmin}^{d} b_m & = 2\sum_{*=2c_B-n}^{2sc_B} b_* \\
& = 2\sum_{*'=2sc_B}^{n+(4s-2)c_B} \sum_{m=1}^\infty \dim \H_{(4s-2m)c_B - *' + n}(B;\Q) \\
& = 2\sum_{*'=2sc_B}^{n+(4s-2)c_B} \sum_{m=1}^{2s-1} \dim \H_{(4s-2m)c_B - *' + n}(B;\Q) \\
& = \sum_{m=1}^{2s-1} \overbrace{\sum_{*'=2sc_B}^{n+(4s-2)c_B} \dim \H_{(4s-2m)c_B - *' + n}(B;\Q) + \dim \H_{n + *' - (4s-2m)c_B }(B;\Q)}^{d_m},
\end{align*}
where the third equality follows from the fact that $2sc_B > 2n$, the last equation holds by Poincar\'e duality and $d_m$ is defined as in the equation.

We have to prove that
\begin{equation}
\label{eq:claim}
\sum_{m=1}^{2s-1} d_m = (2s-1)r_B + \dim \H_n(B;\Q) + 2\sum_{m=s+1}^{2s-1} \dim \H_{n+2(s-m)c_B}(B;\Q).
\end{equation}
In order to do it, let us look at the first term of $d_m$. When $*'$ runs from $2sc_B$ to $(4s-2)c_B+n$, it sums up the ranks of the homology of $B$ from the degree $-2mc_B+2c_B\leq 0$ to $n+2(s-m)c_B$. Similarly, the second term of $d_m$ sums up the ranks of the homology of $B$ from the degree $n-2(s-m)c_B$ to $2n+2(m-1)c_B \geq 2n$.

Hence, we have the following computations of $d_m$ according to the relation between $m$ and $s$:
\begin{itemize}
\item $m<s\ \iff\ 2(s-m)c_B>0$
\begin{equation}
\label{eq:d_m1}
d_m = \sum_{j=0}^{n+2(s-m)c_B} \dim \H_j(B;\Q) + \sum_{j=n-2(s-m)c_B}^{2n}  \dim \H_j(B;\Q) = r_B + \sum_{j=n-2(s-m)c_B}^{n+2(s-m)c_B}  \dim \H_j(B;\Q).
\end{equation}

\item $m=s\ \iff\ 2(s-m)c_B=0$
\begin{equation}
\label{eq:d_m2}
d_m = \sum_{j=0}^{n} \dim \H_j(B;\Q) + \sum_{j=n}^{2n}  \dim \H_j(B;\Q) = r_B + \dim \H_n(B;\Q).
\end{equation}

\item $m>s\ \iff\ 2(s-m)c_B<0$
\begin{align*}
d_m & = \sum_{j=0}^{n+2(s-m)c_B} \dim \H_j(B;\Q) + \sum_{j=n-2(s-m)c_B}^{2n}  \dim \H_j(B;\Q) \\
& = r_B - \sum_{j=n+2(s-m)c_B}^{n-2(s-m)c_B}  \dim \H_j(B;\Q) + \dim \H_{n+2(s-m)c_B}(B;\Q) + \dim \H_{n-2(s-m)c_B}(B;\Q) \\
& = r_B - \sum_{j=n+2(s-m)c_B}^{n-2(s-m)c_B}  \dim \H_j(B;\Q) + 2\dim \H_{n+2(s-m)c_B}(B;\Q), \numberthis \label{eq:d_m3}
\end{align*}
where the last equality holds by Poincar\'e duality.
\end{itemize}
Clearly,
\begin{equation}
\label{eq:d_m4}
\sum_{m=1}^{s-1} \sum_{j=n-2(s-m)c_B}^{n+2(s-m)c_B}  \dim \H_j(B;\Q) = \sum_{m=s+1}^{2s-1} \sum_{j=n+2(s-m)c_B}^{n-2(s-m)c_B}  \dim \H_j(B;\Q).
\end{equation}
Hence, \eqref{eq:claim} follows from \eqref{eq:d_m1}, \eqref{eq:d_m2}, \eqref{eq:d_m3} and \eqref{eq:d_m4}.
\end{proof}

\vskip .3cm
\noindent
{\bf Case 2. $n$ is even}
\vskip .2cm

\vskip .3cm
\noindent
{\bf Step 1. $\alpha$ has at least $(r_B-\dim \H_n(B;\Q))/2$ simple contractible closed orbits}
\vskip .2cm

The identities \eqref{eq:1}, \eqref{eq:3} and \eqref{eq:4} still hold in the case where $n$ is even. However, the other counts will change. We point out that it becomes now possible for an orbit to have index $d$ (since both $d$ and $n$ are even). This can happen in two different ways: either it is an orbit of the form $\ga_i^{k_i}$ or an orbit of the form $\ga_i^{k_i+\ell}$ for some $\ell \in [-\ell_0,\ell_0]\setminus\{0\}$. (Indeed, by the definition of $\ell_0$, $\cz(\ga_i^{k_i\pm\ell})\neq d$ for every $\ell>\ell_0$.) In the last case, $\cz(\ga_i^\ell)=0$ and we have the pair of orbits $(\ga_i^{k_i-\ell},\ga_i^{k_i+\ell})$ with index $d$.

Define $c_{0,i}=\#\{j \in \N;\, \cz(\ga_i^j)=0\}$ and ${\bar c}_{d,i}=\#\{j \in [-\ell_0,\ell_0]\setminus\{0\};\, \cz(\ga_i^{k_i+j})=d\}$. By the previous discussion,
\begin{equation}
\label{eq:5}
{\bar c}_{d,i} = 2c_{0,i}.
\end{equation}
Now, let $a_i=\#\{\ell \in [-\ell_0,\ell_0]\setminus\{0\};\, \cz(\ga_i^{k_i +\ell}) \leq d-1\}$ and $b_i=\#\{\ell \in [-\ell_0,\ell_0]\setminus\{0\};\, \cz(\ga_i^{k_i +\ell}) \geq d+1\}$ be the two terms  in the left hand side of \eqref{eq:2}. Equation \eqref{eq:2} has to be replaced, in this case, with
\begin{equation}
\label{eq:6}
a_i + b_i + {\bar c}_{d,i} = 2\ell_0.
\end{equation}
Using \eqref{eq:1}, \eqref{eq:5} and \eqref{eq:6} we arrive at
\begin{equation}
\label{eq:7}
a_i + c_{0,i} = \ell_0.
\end{equation}
Now, equation \eqref{eq:k_i} becomes
\begin{align*}
\#\{j \in \N,\, j\neq k_i;\, \cz(\ga_i^j) \leq d\} & = \#\{1\leq j \leq k_i+\ell_0,\, j\neq k_i;\, \cz(\ga_i^j) \leq d\} \\
& = k_i - \ell_0 - 1 + \#\{k_i - \ell_0 \leq j \leq k_i+\ell_0,\, j\neq k_i;\, \cz(\ga_i^j) \leq d\} \\
& = k_i - \ell_0 - 1 + a_i + {\bar c}_{d,i} \\
& = k_i - \ell_0 - 1 + a_i + 2c_{0,i} \\
& = k_i - \ell_0 - 1 + \ell_0 + c_{0,i} \\
& = k_i - 1 + c_{0,i} \numberthis \label{eq:k_i2},
\end{align*}
where the first equation follows \eqref{eq:4}, the second equation follows from \eqref{eq:3}, the third equation is immediate from the definitions of $a_i$ and ${\bar c}_{d,i}$, the fourth equality holds by \eqref{eq:5} and the fifth identity is a consequence of \eqref{eq:7}.

Recall that $b_0 = \dim \HC^0_0(M) = \dim \oplus_{m=1}^\infty \H_{n-2mc_B}(B;\Q)$. Since $\alpha$ is lacunary, $\sum_{i=1}^r c_{0,i} = b_0$. Thus, arguing as before, but now using \eqref{eq:k_i2},
\begin{align*}
\sum_{m=\kmin}^{d+1} c_m & = \sum_{m=\kmin}^{d} c_m \\
& = \sum_{i=1}^r \#\{j \in \N;\ \cz(\ga_i^j)\leq d\} \\
& = \sum_{i=1}^r \#\{j \in \N,\, j\neq k_i;\, \cz(\ga_i^j) \leq d\} + \#\{1\leq i \leq r;\, \cz(\ga_i^{k_i})\leq d\} \\
& = \sum_{i=1}^r k_i - r_+ + \sum_{i=1}^r c_{0,i} \\
& = sr_B - r_+ + b_0 \numberthis \label{eq:sumc_m-even-2},
\end{align*}
where the first equality simply follows from the facts that $d$ and $n$ are even and $\alpha$ is lacunary. This is similar to \eqref{eq:sumc_m-even} but with the correction term $b_0$.

Note that taking $s$ big enough such that $2(1-s)c_B < -n$ we have that
\[
\sum_{m=s+1}^{2s-1} \dim \H_{n+2(s-m)c_B}(B;\Q) = b_0.
\]
Consequently, we conclude from Lemma \ref{lemma:sumb_m} that
\begin{equation}
\label{eq:sumb_m-even-2}
\sum_{m=\kmin}^{d+1} b_m = \sum_{m=\kmin}^{d} b_m = sr_B +b_0 - (r_B - \dim \H_n(B;\Q))/2,
\end{equation}
where the first identity holds, again, by the fact $d$ and $n$ are even. As before, this is the analogous of \eqref{eq:sumb_m-even} but with the correction term $b_0$.

Then, the first step follows from \eqref{eq:sumc_m-even-2} and \eqref{eq:sumb_m-even-2} as in Section \ref{sec:proof main}.

\vskip .2cm
\noindent
{\bf Step 2} is modified according to {\bf Step 1}.

\vskip .3cm
\noindent
{\bf Step 3. Existence of more $\dim \H_n(B;\Q)$ simple contractible closed orbits}
\vskip .2cm

As in Section \ref{sec:proof main}, from {\bf Step 1} and {\bf Step 2} we get $r_B - \dim \H_n(B;\Q)$ closed orbits given by the $\ga_i's$ such that $\cz(\ga_i^{k_i})<d - 1$ and $\cz(\ga_i^{k_i})>d + 1$. As mentioned earlier, the set of contractible orbits with index $d$ splits into two sets: orbits of the form $\ga_i^{k_i}$ such that $\cz(\ga_i^{k_i})=d$ and pairs of orbits $(\ga_i^{k_i-\ell},\ga_i^{k_i+\ell})$ with $\cz(\ga_i^\ell)=0$. The former set has cardinality $r-r_--r_+$ and the latter has cardinality $2b_0$. So we get
\[
c_d=2b_0+r-r_--r_+.
\] 

On the other hand, arguing as before and using Poincar\'e duality,
\begin{align*}
b_d & = \dim \bigoplus_{m=1}^\infty \H_{n+2(s-m)c_B}(B;\Q) \\
& = \sum_{m=s+1}^{2s-1} \dim \H_{n+2(s-m)c_B}(B;\Q) + \sum_{m=1}^{s-1} \dim \H_{n+2(s-m)c_B}(B;\Q) + \dim \H_n(B;\Q) \\
& = 2b_0 + \dim \H_n(B;\Q).
\end{align*}

Thus, since $c_d=b_d$, we have that $r-r_--r_+=\dim \H_n(B;\Q)$. It concludes the proof of {\bf Step 3}.

Finally, we finish the proof as in {\bf Step 4} of Section \ref{sec:proof main}.

\end{document}